\documentclass[12pt]{amsart}
\usepackage[letterpaper,margin=1.1in]{geometry}
\usepackage{graphicx}
\usepackage{esint}
\usepackage{amssymb}
\usepackage[mathscr]{euscript}
\usepackage{enumerate}

\newtheorem{theorem}{Theorem}[section]
\newtheorem{lemma}[theorem]{Lemma}
\newtheorem{corollary}[theorem]{Corollary}

\newtheorem*{theorem*}{Theorem}

\theoremstyle{definition}
\newtheorem{definition}[theorem]{Definition}

\theoremstyle{remark}
\newtheorem{remark}[theorem]{Remark}
\newtheorem{example}[theorem]{Example}

\newtheorem*{ack}{Acknowledgements}

\let\inf\relax \DeclareMathOperator*\inf{\vphantom{p}inf}

\newcommand{\VMO}{\mathrm{VMO}}
\newcommand{\BMO}{\mathrm{BMO}}

\newcommand{\cF}{\mathcal{F}}
\newcommand{\starF}{\mathcal{F}^*}

\newcommand{\cH}{\mathcal{H}}
\newcommand{\starH}{\mathcal{H}^*}

\newcommand{\cS}{\mathcal{S}}
\newcommand{\cT}{\mathcal{T}}
\newcommand{\cTp}{{\mathcal{T}^\perp}}

\newcommand{\hz}{\widehat{\zeta}}

\newcommand{\RR}{\mathbb{R}}
\newcommand{\CC}{\mathbb{C}}
\newcommand{\ccS}{\overline{\mathcal{S}}}
\newcommand{\ccT}{{\overline{\mathcal{T}}}}

\newcommand{\ball}{B}

\newcommand{\ex}{\mathop\mathrm{excess}\nolimits}
\newcommand{\gap}{\mathop\mathrm{gap}\nolimits}
\newcommand{\dist}{\mathop\mathrm{dist}\nolimits}

\newcommand{\diam}{\mathop\mathrm{diam}\nolimits}
\newcommand{\Tan}{\mathop\mathrm{Tan}\nolimits}

\newcommand{\PsTan}{\mathop{\Psi\mbox{-}\mathrm{Tan}}\nolimits}

\newcommand{\D}[2]{\mathop{\mathrm{D}^{#1,#2}}\nolimits}

\newcommand{\mD}[2]{\mathop{\widetilde{\mathrm{D}}^{#1,#2}}\nolimits}
\newcommand{\mud}[2]{\mathop{\widetilde{\mathrm{d}}^{\,#1, #2}}\nolimits}
\newcommand{\covplain}{\mathop\mathrm{N}\nolimits}

\newcommand{\udim}{\mathop{\overline{\mathrm{dim}}}\nolimits}

\newcommand{\sign}{\mathop\mathrm{sign}\nolimits}
\newcommand{\cl}[1]{{\overline{#1}}}
\newcommand{\sing}[1]{\mathop\mathrm{sing}_{#1}\nolimits}
\newcommand{\Vol}{\mathop\mathrm{Vol}\nolimits}

\newcommand{\CL}[1]{\mathfrak{C}(#1)}

\numberwithin{equation}{section}
\numberwithin{figure}{section}

\begin{document}

\title[Sets well approximated by zero sets of harmonic polynomials]{Structure of sets which are well approximated by zero sets of harmonic polynomials}
\date{February 1, 2017}
\author{Matthew Badger}
\author{Max Engelstein}
\author{Tatiana Toro}
\thanks{M.~Badger was partially supported by NSF grant DMS 1500382. M.~Engelstein was partially supported by an NSF Graduate Research Fellowship, NSF DGE 1144082. T.~Toro was partially supported by NSF grant DMS 1361823, and the Robert R.~\& Elaine F.~Phelps Professorship in Mathematics.}
\subjclass[2010]{ Primary 33C55, 49J52. Secondary 28A75, 31A15, 35R35}
\keywords{Reifenberg type sets, harmonic polynomials, \L ojasiewicz type inequalities, singular set, Hausdorff and Minkowski dimensions, two-phase free boundary problems, harmonic measure, NTA domains}
\address{Department of Mathematics\\ University of Connecticut\\ Storrs, CT 06269-3009}
\email{matthew.badger@uconn.edu}
\address{Department of Mathematics\\Massachusetts Institute of Technology\\Cambridge, MA, 02139-4307 }
\email{maxe@mit.edu}
\address{Department of Mathematics\\ University of Washington\\ Box 354350\\ Seattle, WA 98195-4350}
\email{toro@uw.edu}

\begin{abstract} The zero sets of harmonic polynomials play a crucial role in the study of the free boundary regularity problem for harmonic measure. In order to understand the fine structure of these free
boundaries a detailed study of the singular points of these zero sets is required. In this paper we study how ``degree $k$ points" sit inside zero sets of harmonic polynomials in $\RR^n$ of degree $d$ (for all $n\geq 2$ and $1\leq k\leq d$) and inside sets that admit arbitrarily good local approximations by zero sets of harmonic polynomials. We obtain a general structure theorem for the latter type of sets, including sharp Hausdorff and Minkowski dimension estimates on the singular set of ``degree $k$ points" ($k\geq 2$) without proving uniqueness of blowups or aid of PDE methods such as monotonicity formulas. In addition, we show that in the presence of a certain topological separation condition, the sharp dimension estimates improve and depend on the parity of $k$. An application is given to the two-phase free boundary regularity problem for harmonic measure below the continuous threshold introduced by Kenig and Toro.\end{abstract}

\maketitle

\setcounter{tocdepth}{1}
\tableofcontents

\section{Introduction}\label{sec:intro}

In this paper, we study the geometry of sets that admit arbitrarily good local approximations by zero sets of harmonic polynomials. As our conditions are reminiscent of those introduced by Reifenberg
\cite{Reifenberg}, we often refer to these sets as Reifenberg type sets which are well approximated by zero sets of harmonic polynomials.
This class of sets plays a crucial role in the study of a two-phase free boundary problem for harmonic measure with weak initial regularity, examined first by Kenig and Toro \cite{kenigtorotwophase}
and subsequently by
Kenig, Preiss and Toro \cite{kenigpreisstoro}, Badger \cite{badgerharmonicmeasure,badgerflatpoints}, Badger and Lewis \cite{localsetapproximation}, and Engelstein \cite{engelsteintwophase}.
Our results are partly motivated by several open questions about the structure and size of the singular set in the free boundary, which we answer definitively below.
In particular, we obtain sharp bounds on the upper Minkowski and Hausdorff dimensions of the singular set, which depend on the degree of blowups of the boundary.
It is important to remark that this is one of those rare instances in which a singular set of a non-variational problem can be well understood. Often, in this type of question, the lack of a monotonicity formula
is a serious obstacle.
A remarkable feature of the proof is that
\L ojasiewicz type inequalities for harmonic polynomials
are used to establish a relationship between the terms in the Taylor expansion of a harmonic polynomial at a given point in its zero set and the extent to which this zero
set can be approximated by the zero set of a lower order harmonic polynomial (see \S\S 3 and 4).
In a broader context, this paper also complements the recent investigations by Cheeger, Naber, and Valtorta \cite{cheegernabervaltorta} and Naber and Valtorta \cite{nabervaltorta} into volume estimates for the critical sets of harmonic functions and solutions to certain  second-order elliptic operators with Lipschitz coefficients. Detailed descriptions of these past works and new results appear below, after we introduce some requisite notation.

For all $n\geq 2$ and $d\geq 1$, let $\cH_{n,d}$ denote the collection of all zero sets $\Sigma_p$ of nonconstant harmonic polynomials $p:\RR^n\rightarrow\RR$ of degree at most $d$ such that $0\in\Sigma_p$ (i.e.~$p(0)=0$). For every nonempty set $A\subseteq\RR^n$, location $x\in A$, and scale $r>0$, we introduce the \emph{bilateral approximation number} $\Theta^{\cH_{n,d}}_A(x,r)$, which, roughly speaking, records how well $A$ looks like some zero set of a harmonic polynomial of degree at most $d$ in the open ball
$B(x,r)=\{y\in\RR^n:|y-x|<r\}$:
\begin{equation}\label{bil-app-num}
 \Theta^{\cH_{n,d}}_A(x,r) = \frac{1}{r}\inf_{\Sigma_p\in\cH_{n,d}}\max\left\{ \sup_{a\in A\cap B(x,r)} \dist(a,x+\Sigma_p),\  \sup_{z\in (x+\Sigma_p)\cap B(x,r)}\dist(z,A)\right\}\in[0,1].
\end{equation}
 When $\Theta^{\cH_{n,d}}_A(x,r)=0$, the closure, $\overline{A}$, of $A$ coincides with the zero set of some harmonic polynomial of degree at most $d$ in $B(x,r)$. At the other extreme, when $\Theta^{\cH_{n,d}}_A(x,r)\sim 1$, the set $A$ stays ``far away" in $B(x,r)$ from every zero set of a nonconstant harmonic polynomial of degree at most $d$ containing $x$. We observe that the approximation numbers are scale-invariant in the sense that $\Theta^{\cH_{n,d}}_{\lambda A}(\lambda x,\lambda r) = \Theta^{\cH_{n,d}}_A(x,r)$ for all $\lambda>0$.
A point $x$ in a nonempty set $A$ is called an \emph{$\cH_{n,d}$ point} of $A$ if $\lim_{r\rightarrow 0} \Theta^{\cH_{n,d}}_A(x,r)=0$.

For all $n\geq 2$ and $k\geq 1$, let $\cF_{n,k}$ denote the collection of all zero sets of homogeneous harmonic polynomials $p:\RR^n\rightarrow\RR$ of degree $k$. We note that $$\cF_{n,k}\subseteq \cH_{n,d}\quad\text{whenever }1\leq k\leq d.$$ For every nonempty set $A\subseteq\RR^n$, $x\in A$, and $r>0$, the bilateral approximation number $\Theta_A^{\cF_{n,k}}(x,r)$ is defined analogously to $\Theta_A^{\cH_{n,d}}(x,r)$ except that the zero set $\Sigma_p$ in the infimum ranges over $\cF_{n,k}$ instead of $\cH_{n,d}$. A point $x$ in a nonempty set $A$ is called an \emph{$\cF_{n,k}$ point} of $A$ if $\lim_{r\rightarrow 0} \Theta^{\cF_{n,k}}_A(x,r)=0$. This means that infinitesimally at $x$, $A$ looks like the zero set of a homogeneous harmonic polynomial of degree $k$.

We say that a nonempty set $A\subseteq\RR^n$ is \emph{locally bilaterally well approximated by $\cH_{n,d}$} if for all $\varepsilon>0$ and for all compact sets $K\subseteq A$ there exists $r_{\varepsilon,K}>0$ such that $\Theta_A^{\cH_{n,d}}(x,r)\leq \varepsilon$ for all $x\in K$ and $0<r\leq r_{\varepsilon,K}$. When $k=1$, $\cH_{n,1}=\cF_{n,1}=G(n,n-1)$ is the collection of codimension 1 hyperplanes through the origin and sets $A$ that are locally bilaterally well approximated by $\cH_{n,1}$ are also called \emph{Reifenberg flat sets with vanishing constant} or \emph{Reifenberg vanishing sets} (e.g., see \cite{davidkenigtoro}). Our initial result is the following structure theorem for sets that are locally bilaterally well approximated by $\cH_{n,d}$.

\begin{theorem}\label{t:main1} Let $n\geq 2$ and $d\geq 2$. If $A\subseteq\RR^n$ is locally bilaterally well approximated by $\cH_{n,d}$, then we can write $A$ as a disjoint union, \begin{equation*}\label{e:a123} A=A_1\cup\dots\cup A_{d}\quad(i\neq j\Rightarrow A_i\cap A_j=\emptyset),\end{equation*} with the following properties. \begin{enumerate}
\item For all $1\leq k\leq d$,  $x\in A_k$ if and only if $x$ is an $\cF_{n,k}$ point of $A$.
\item For all $1\leq k\leq d$, the set $U_k:=A_1\cup\dots \cup A_k$ is relatively open in $A$.
\item For all $1\leq k\leq d$, $U_k$ is locally bilaterally well approximated by $\cH_{n,k}$.
\item For all $2\leq k\leq d$, $A$ is locally bilaterally well approximated along $A_k$ by $\cF_{n,k}$, i.e.~$\limsup_{r\downarrow 0} \sup_{x\in K} \Theta_A^{\cF_{n,k}}(x,r)=0$ for every compact set $K\subseteq A_k$.
\item For all $1\leq l<k\leq d$, $U_l$ is relatively open in $U_k$ and $A_{l+1}\cup\dots \cup A_k$ is relatively closed in $U_k$.
\item The set $A_1$ is relatively dense in $A$, i.e.~ $\overline{A_1}\cap A=A$.
\end{enumerate} If, in addition, $A$ is closed and nonempty, then
\begin{enumerate}
\item[(vii)] $A$ has upper Minkowski dimension and Hausdorff dimension $n-1$; and,
\item[(viii)] $A\setminus A_1=A_2\cup\dots \cup A_{d}$ has upper Minkowski dimension at most $n-2$.
\end{enumerate}
\end{theorem}

\begin{remark} If $\Sigma_p\in\cH_{n,d}$, then $\Sigma_p$ is locally bilaterally well approximated by $\cH_{n,d}$, simply because $\Theta^{\cH_{n,d}}_{\Sigma_p}(x,r)=0$ for all $x\in \Sigma_p$ and $r>0$. Since $A=\Sigma_p$ corresponding to  $p(x_1,\dots,x_n)=x_1x_2$ has $A_2=\{0\}^2\times\RR^{n-2}$, we see that the dimension bounds on $A\setminus A_1$ in Theorem \ref{t:main1} hold by example, and thus, are generically the best possible.\end{remark}

\begin{remark}\label{r:reif} Note that $A_1$ is nonempty if $A$ is nonempty by (vi), $A_1$ is locally closed if $A$ is closed by (ii), and $A_1$ is locally Reifenberg flat with vanishing constant by (iii). Therefore, by Reifenberg's topological disk theorem (e.g., see \cite{Reifenberg} or \cite{davidtoro-holes}), $A_1$ admits local bi-H\"older parameterizations by open subsets of $\RR^{n-1}$ with bi-H\"older exponents arbitrarily close to 1 provided that $A$ is closed and nonempty. However, we emphasize that while $A_1$ always has Hausdorff dimension $n-1$ under these conditions, $A_1$ may potentially have locally infinite $(n-1)$-dimensional Hausdorff measure or may even be purely unrectifiable (e.g., see \cite{davidtoro-snow}).\end{remark}

The proof of Theorem \ref{t:main1} uses a general structure theorem for Reifenberg type sets, developed in \cite{localsetapproximation}, as well as uniform Minkowski content estimates for the zero and singular sets of harmonic polynomials from \cite{nabervaltorta}.  A Reifenberg type set is a set $A\subseteq\RR^n$ that admits uniform local bilateral approximations by sets in a cone $\cS$ of model sets in $\RR^n$. In the present setting, the role of the model sets $\cS$ is played by $\cH_{n,d}$.  For background on the theory of local set approximation and summary of results from \cite{localsetapproximation}, we refer the reader to Appendix \ref{sect:2}. The core geometric result at the heart of Theorem \ref{t:main1} is the following property of zero sets of harmonic polynomials: $\cH_{n,k}$ points can be detected in zero sets of harmonic polynomials of degree $d$ $(1\leq k\leq d)$ by finding a \emph{single}, sufficiently good approximation at a coarse scale. The precise statement is as follows.

\begin{theorem}\label{t:main2} For all $n\geq 2$ and $1\leq k<d$, there exists a constant $\delta_{n,d,k}>0$, depending only on $n$, $d$, and $k$, such that for any harmonic polynomial $p:\RR^n\rightarrow \RR$ of degree $d$ and, for any $x\in \Sigma_p$, \begin{align*} \partial^\alpha p(x)&=0\quad \text{ for all }|\alpha|\leq k  &&\Longleftrightarrow &&\Theta_{\Sigma_p}^{\cH_{n,k}}(x,r) \geq \delta_{n,d,k}\quad \text{for all }r>0,\\
\partial^\alpha p(x)&\neq 0\quad\text{ for some }|\alpha|\leq k &&\Longleftrightarrow &&\Theta_{\Sigma_p}^{\cH_{n,k}}(x,r) < \delta_{n,d,k}\quad\text{for some }r>0.\end{align*} Moreover, there exists a constant $C_{n,d,k}>1$ depending only on $n$, $d$, and $k$ such that
\begin{equation}\begin{split} \label{e:detect}
\Theta_{\Sigma_p}^{\cH_{n,k}}(x,r)<\delta_{n,d,k}&\text{ for some } r>0\\
&\Longrightarrow\ \Theta_{\Sigma_p}^{\cH_{n,k}}(x,sr)< C_{n,d,k}\,s^{1/k}\text{ for all } s\in(0,1).
\end{split}\end{equation}
\end{theorem}

In particular, applying \eqref{e:detect} with $\Sigma_p\in\cH_{n,d}$ and $x=0$, we obtain the following property.

\begin{corollary}\label{c:main2} In the language of Definition \ref{d:Tpdp}, $\cH_{n,k}$ points are detectable in $\cH_{n,d}$.
\end{corollary}

\begin{remark} The reader may recognize \eqref{e:detect} as an ``improvement type lemma'', which is often obtained as a consequence of a monotonicity formula or a blow-up argument. Here this improvement result states that
at every $\cH_{n,k}$ point in the zero set $\Sigma_p$ of a harmonic polynomial of degree $d>k$, the zero set $\Sigma_p$ resembles the zero set of a harmonic polynomial of degree at most $k$ at scale
$r$ \emph{with increasing certainty}  as $r\downarrow 0$. In fact, \eqref{e:detect} yields a precise rate of convergence for the approximation number $\Theta_{\Sigma_p}^{\cH_{n,k}}(x,sr)$ as $s$ goes to 0 provided
$\Theta_{\Sigma_p}^{\cH_{n,k}}(x,r)$ is small enough. However, we would like to emphasize that the proof of
Theorem \ref{t:main1} does not require monotone convergence nor a definite rate of convergence of the blowups $(A-x)/r$ of the set $A$ as $r\downarrow 0$. Rather, the proof of Theorem \ref{t:main1} relies only on the fact that the
\emph{pseudotangents} $T=\lim_{i\rightarrow\infty}(A-x_i)/t_i$ of $A$ at $x$ (along sequences $x_i\rightarrow x$ in $A$ and $t_i\downarrow 0$) satisfy \eqref{e:detect}. The authors expect that both this
``improvement type lemma"  as well as the way in which it is applied in the proof of Theorem \ref{t:main1}
should be
useful in other situations where questions about the structure and size of sets with singularities arise. \end{remark}

In the special case when $k=1$, Theorem \ref{t:main2} first appeared in \cite[Theorem 1.4]{badgerflatpoints}. The proof of the general case, given in \S\S 2--4 below, follows the same guidelines, but requires more sophisticated estimates. In particular, in \S3, we establish uniform growth and size estimates for harmonic polynomials of bounded degree. Of some note, we prove that harmonic polynomials of bounded degree satisfy a \L ojasiewicz type inequality with uniform constants (see Theorem \ref{lojasiewiczinequality}). These estimates are essential to show that the approximability $\Theta^{\cH_{n,k}}_{\Sigma_p}(x,r)$ of a zero set $\Sigma_p\in\cH_{n,d}$ is controlled from above by the relative size $\widehat{\zeta}_k(p,x,r)$  of the terms of degree at most $k$ appearing in the Taylor expansion for $p$ at $x$ (see Definition \ref{modifiedzeta} and Lemma \ref{thetacontrolledbyzeta}).

Applied to harmonic polynomials of degree at most $d$, \cite[Theorem A.3]{nabervaltorta} says that
\begin{equation}\label{e:nv1}\Vol\big(\{x\in B(0,1/2):\dist(x,\Sigma_p)\leq r\}\big)\leq (C(n)d)^d\,r\quad\text{for all }\Sigma_p\in\cH_{n,d},\end{equation} and  \cite[Theorem 3.37]{nabervaltorta} says that
\begin{equation}\label{e:nv2}\Vol\big(\{x\in \ball(0,1/2): \dist(x,S_p)\leq r\}\big) \leq C(n)^{d^2}r^2\quad\text{for all }S_p\in\cS\cH_{n,d},\end{equation} where $\cS\cH_{n,d}=\{S_p=\Sigma_p\cap|Dp|^{-1}(0):\Sigma_p\in\cH_{n,d}, 0\in S_p\}$ denotes the collection of singular sets of nonconstant harmonic polynomials in $\RR^n$ of degree at most $d$ that include the origin. The latter estimate is a refinement of  \cite{cheegernabervaltorta}, which gave bounds on the volume of the $r$-neighborhood of the singular set of the form $C(n,d,\varepsilon)r^{2-\varepsilon}$ for all $\varepsilon>0$. The results of Cheeger, Naber, and Valtorta \cite{cheegernabervaltorta} and Naber and Valtorta \cite{nabervaltorta} apply to solutions of a class of second-order elliptic operators with Lipschitz coefficients; we refer the reader to the original papers for the precise class. Estimates \eqref{e:nv1} and \eqref{e:nv2} imply that the zero sets and the singular sets of harmonic polynomials have locally finite $(n-1)$ and $(n-2)$ dimensional Hausdorff measure, respectively. They transfer to the dimension estimates in Theorem \ref{t:main1} for sets that are locally bilaterally well approximated by $\cH_{n,d}$ using \cite{localsetapproximation}. See the proof of Theorem \ref{t:main1} in \S5 for details.

Although the singular set of a harmonic polynomial in $\RR^n$ generically has dimension at most $n-2$, additional topological restrictions on the zero set may lead to better bounds. In the plane, for example, the zero set of a homogeneous harmonic polynomial of degree $k$ is precisely the union of $k$ lines through the origin, arranged in an equiangular pattern. Hence $\RR^2\setminus \Sigma_p$ has precisely two connected components for $\Sigma_p\in\cF_{2,k}$ if and only if $k=1$, and consequently, the singular set is empty for any harmonic polynomial whose zero set separates $\RR^2$ into two connected components. When $n=3$, Lewy \cite{lewy} proved that if $\RR^3\setminus\Sigma_p$ has precisely two connected components for $\Sigma_p\in\cF_{3,k}$, then $k$ is necessarily odd. Moreover, Lewy proved the existence of $\Sigma_p\in\cF_{3,k}$ that separate $\RR^3$ into two connected components for all odd $k\geq 3$; an explicit example due to Szulkin \cite{szulkin} is $\Sigma_p\in\cF_{3,3}$, where $$p(x,y,z)=x^3-3xy^2+z^3-\tfrac{3}{2}(x^2+y^2)z.$$ Starting with $n=4$, zero sets of even degree homogeneous harmonic polynomials can also separate $\RR^n$ into two components, as shown e.g.~by Lemma \ref{l:example}, which we prove in \S6.

\begin{lemma}\label{l:example} Let $k\geq 2$, even or odd, and let $q:\RR^2\rightarrow\RR$ be a homogeneous harmonic polynomial of degree $k$. For any pair of constants $a,b\neq 0$, consider the  homogeneous harmonic polynomial $p:\RR^4\rightarrow\RR$ of degree $k$ given by
$$p(x_1,y_1,x_2,y_2)=a\, q(x_1,y_1)+b\,q(x_2,y_2).$$
The zero set $\Sigma_p$ of $p$ separates $\RR^4$ into two components.\end{lemma}

Motivated by these examples, it is natural to ask whether it is possible to improve the dimension bounds on the singular set $A\setminus A_1=A_2\cup\dots\cup A_d$ in Theorem \ref{t:main1} under additional topological restrictions on $A$. In this direction, we prove the following result in \S6 below.

\begin{theorem}\label{t:main3} Let $n\geq 2$ and $d\geq 2$. Let $A\subseteq\RR^n$ be a closed set that is locally bilaterally well approximated by $\cH_{n,d}$. If $\RR^n\setminus A=\Omega^+\cup\Omega^-$ is a union of complimentary NTA domains $\Omega^+$ and $\Omega^-$, then \begin{enumerate}
\item[(i)]  $A\setminus A_1=A_2\cup\dots \cup A_{d}$ has upper Minkowski dimension at most $n-3$;
\item[(ii)] The ``even singular set" $A_2\cup A_4\cup A_6 \cup \cdots$ has Hausdorff dimension at most $n-4$.
\end{enumerate}\end{theorem}

NTA domains, or \emph{non-tangentially accessible domains}, were introduced by Jerison and Kenig \cite{jerisonandkenig} to study the boundary behavior of harmonic functions in dimensions three and above. We defer their definition to \S6. However, let us mention in particular that NTA domains satisfy a quantitative strengthening of path connectedness called the \emph{Harnack chain condition}. This property guarantees that $A$ appearing in Theorem \ref{t:main3} may be locally bilaterally well approximated by zero sets $\Sigma_p$ of harmonic polynomials such that $\RR^n\setminus\Sigma_p$ has two connected components. Without the Harnack chain condition, this property may fail due to the following example by Logunov and Malinnikova \cite{lm-ratios}.

\begin{example} Consider the harmonic polynomial $p(x,y,z)=x^2-y^2+z^3-3x^2z$ from \cite[Example 5.1]{lm-ratios}. The authors of \cite{lm-ratios} show that $\RR^n\setminus\Sigma_p=\Omega^+\cup\Omega^-$ is the union of two domains, but remark that $\Omega^+$ and $\Omega^-$ fail the Harnack chain condition, and thus, $\Omega^+$ and $\Omega^-$ are not NTA domains (see Figure 1.1). Using Lemma \ref{convergenceofzerosets} below, it can be shown that $\Sigma_p$ has a unique tangent set at the origin (see Definition \ref{d:tangent} in the appendix), given by $\Sigma_q$, where $q(x,y,z)=x^2-y^2$. Note that $\Sigma_q$ divides $\RR^3$ into four components. However, if the set $\Sigma_p$ is locally bilaterally well approximated by some closed class $\cS\subseteq\cH_{n,d}$, then $\Sigma_q\in\cS$ by Theorem \ref{t:tangent-well} below.
\begin{figure}
\begin{center}\includegraphics[height=2in]{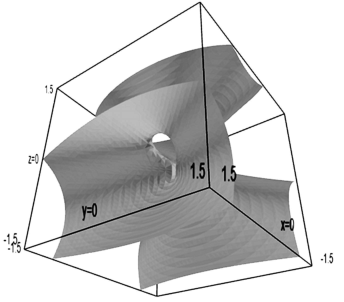}\hspace{.25in}\includegraphics[height=2in]{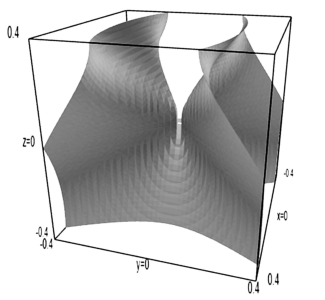}\end{center}
\caption{Select views of $\Sigma_p$, $p(x,y,z)=x^2-y^2+z^3-3x^2z$, which separates $\RR^3$ into two components \emph{and} has a cusp at the origin.}
\end{figure}
\end{example}

\begin{remark}\label{r:sharp} It can be shown that $\RR^n\setminus\Sigma_p=\Omega^+\cup\Omega^-$ is a union of complementary NTA domains and $\Sigma_p$ is smooth except at the origin when $p(x,y,z)$ is Szulkin's polynomial or when $p(x_1,y_1,x_2,y_2)$ is any polynomial from Lemma \ref{l:example}. Thus, the upper bounds given in Theorem \ref{t:main3} are generically the best possible. The reason that we obtain an upper Minkowski dimension bound on the full singular set $A\setminus A_1$, but only obtain a Hausdorff dimension bound on the ``even" singular set $A_2\cup A_4\cup\cdots$ is that the former is always closed when $A$ is closed, but we only know that the latter is $F_\sigma$ when $A$ is closed (see the proof of Theorem \ref{t:main3}). \end{remark}

The improved dimension bounds on $A\setminus A_1$ in Theorem \ref{t:main3} require a refinement of \eqref{e:nv2} for $\Sigma_p\in\cH_{n,d}$ that separate $\RR^n$ into complementary NTA domains, whose existence was postulated in \cite[Remark 9.5]{localsetapproximation}. Using the quantitative stratification machinery introduced in \cite{cheegernabervaltorta}, we demonstrate that near its singular points a zero set $\Sigma_p\in\cH_{n,d}$ with the separation property does not resemble $\Sigma_h\times\RR^{n-2}$ for any $\Sigma_h\in\cF_{2,k}$, $2\leq k\leq d$. This leads us to a version of \eqref{e:nv2} with right hand side $C(n,d,\varepsilon) r^{3-\varepsilon}$ for all $\varepsilon>0$ and thence to $\udim_M A\setminus A_1\leq n-3$ using \cite{localsetapproximation}. In addition, we show that at ``even degree" singular points, a zero set $\Sigma_p$ with the separation property, does not resemble $\Sigma_h\times\RR^{n-3}$ for any $\Sigma_h\in\cF_{3,2k}$, $2\leq 2k\leq d$. This leads us to the  bound $\dim_H \Gamma_2\cup\Gamma_4\cup\cdots \leq n-4$. See the proof of Theorem \ref{t:main3} in \S6 for details.

In the last section of the paper, \S7, we specialize Theorem \ref{t:main1} and Theorem \ref{t:main3} to the setting of two-phase free boundary problems for harmonic measure mentioned above, which motivated our investigation. This includes the case that $A=\partial\Omega$ is the boundary of a 2-sided NTA domain $\Omega\subset\RR^n$ whose interior harmonic measure $\omega^+$ and exterior harmonic measure $\omega^-$ are mutually absolutely continuous and have Radon-Nikodym derivative $f=d\omega^-/d\omega^+$ satisfying $\log f\in C(\partial\Omega)$ or $\log f\in \VMO(d\omega^+)$.

\begin{ack}  A portion of this research was completed while the second author was visiting the University of Washington during the spring of 2015. He thanks the Mathematics Department at UW for their hospitality. The first author acknowledges and thanks Stephen Lewis for many insightful conversations about local set approximation, which have duly influenced the present manuscript. The authors would like to thank an anonymous referee for his or her critical feedback, which has led to an improved exposition of these results.\end{ack}

\section{Relative size of the low order part of a polynomial}
\label{sect:relative}

Given a polynomial $p(x) =\sum_{|\alpha|\leq d} c_\alpha x^\alpha$ in $\RR^n$, define the \emph{height} $H(p)=\max_{|\alpha|\leq d}|c_\alpha|$, i.e.\ the height of $p$ is the maximum in absolute value of the coefficients of $p$. The following lemma is an instance of the equivalence of norms on finite-dimensional vector spaces.

\begin{lemma}\label{l:height} $H(p)\approx\|p\|_{L^\infty(B(0,1))}$ for every polynomial $p:\RR^n\rightarrow\RR$ of degree at most $d$, where the implicit constants depend only on $n$ and $d$.\end{lemma}

Below we will need the following easy consequence of Lemma \ref{l:height}.

\begin{corollary}\label{c:height} If $p\equiv p_d +\dots+ p_0$, where each $p_i:\RR^n\rightarrow\RR$ is zero or a homogeneous polynomial of degree $i$, then $\|p\|_{L^\infty(B(0,1))}\approx \sum_{i=0}^d H(p_i)$, where the implicit constants depend only on $n$ and $d$. \end{corollary}
\begin{proof} On one hand, $$\|p\|_{L^\infty(B(0,1))} \leq \sum_{i=0}^d \|p_i\|_{L^\infty(B(0,1))} \lesssim \sum_{i=0}^d H(p_i)$$ by Lemma \ref{l:height} (applied $d+1$ times). On the other hand, the assumption that each $p_i$ is zero or homogeneous of degree $i$ ensures that $H(p)=\max_i H(p_i)$. Hence $$\sum_{i=0}^d H(p_i) \leq (d+1) H(p) \lesssim \|p\|_{L^\infty(B(0,1))}$$ by Lemma \ref{l:height}, again. \end{proof}

By Taylor's theorem, for any polynomial $p:\RR^n\rightarrow\RR$ of degree $d\geq 1$ and for any $x\in\RR^n$, we can write \begin{equation} p(x+y) = p^{(x)}_d(y) + p^{(x)}_{d-1}(y)+\dots + p^{(x)}_0(y)\quad\text{for all } y\in\RR^n,\end{equation} where each term $p^{(x)}_i:\RR^n\rightarrow\RR$ is an $i$-homogeneous polynomial, i.e.~ \begin{equation} p^{(x)}_i(ry)=r^i p^{(x)}_i(y)\quad\text{for all } y\in\RR^n \text{ and } r>0.\end{equation}

\begin{definition}\label{modifiedzeta}
Let $p:\RR^n\rightarrow\RR$ be a polynomial of degree $d\geq 1$ and let $x\in\RR^n$. For all $0\leq k<d$ and $r>0$, define $\widehat{\zeta}_k(p,x,r)$ by \begin{equation*}\label{eq:modifiedzeta} \widehat{\zeta}_k(p,x, r) =\max_{k<j\leq d} \frac{\left\|p^{(x)}_j\right\|_{L^\infty(B(0,r))}}{\left\|\sum_{i=0}^k p^{(x)}_i\right\|_{L^\infty(B(0,r))}}\in[0,\infty].
\end{equation*}
\end{definition}

\begin{remark} The function $\widehat{\zeta}_k(p,x,r)$ is a variant of the function $\zeta_{k}(p,x,r)$ appearing in \cite[Definition 2.1]{badgerflatpoints} and defined by $$\zeta_k(p,x,r) = \max_{j\neq k} \frac{\left\|p^{(x)}_j\right\|_{L^\infty(B(0,r))}}{\left\|p^{(x)}_k\right\|_{L^\infty(B(0,r))}}.$$ The latter measured the relative size of the degree $k$ part of a polynomial compared to its parts of degree $j\neq k$, while the former measures the relative size of the low order part of a polynomial, consisting of all terms of degree at most $k$, compared to its parts of degree $j>k$. We note that $\widehat{\zeta}_1(p,x,r)$ and $\zeta_1(p,x,r)$ coincide whenever $x\in\Sigma_p$, the zero set of $p$.
\end{remark}

The next lemma generalizes \cite[Lemma 2.10]{badgerflatpoints}, which stated $\zeta_1(p,x,sr) \leq s\zeta_1(p,x,r)$ for all $s\in(0,1)$, for all polynomials $p:\RR^n\rightarrow\RR$, for all $x\in \Sigma_p$, and for all $r>0$.

\begin{lemma}[change of scales lemma] \label{l:quasilinear} For all polynomials $p:\RR^n\rightarrow\RR$ of degree $d\geq 1$, for all $0\leq k<d$, for all $x\in\RR^n$ and for all $r>0$, $$s^d\,\widehat{\zeta}_k(p,x,r)\lesssim \widehat{\zeta}_k(p,x,sr)\lesssim s\, \widehat{\zeta}_k(p,x,r)\quad \text{for all }s\in(0,1),$$ where the implicit constants depends only on $n$ and $d$.\end{lemma}

\begin{proof} Let $p:\RR^n\rightarrow\RR$ be a polynomial of degree $d\geq 1$, let $x\in\RR^n$, and let $0\leq k<d$. Write $\tilde p=p^{(x)}_k+\dots + p^{(x)}_0$ for the low order part of $p$ at $x$. Then, by repeated use of Corollary \ref{c:height} and the $i$-homogenity of each $p^{(x)}_i$, we have that for all $r>0$ and $s\in(0,1)$, \begin{equation}\begin{split} \label{e:polyshrink} \|\tilde p\|_{L^\infty(B(0,sr))} &= \left\|\sum_{i=0}^k p_i^{(x)}(sr\cdot)\right\|_{L^\infty(B(0,1))}
\gtrsim \sum_{i=0}^k H(p_i^{(x)}(sr\cdot))\gtrsim \sum_{i=0}^k s^i H(p^{(x)}_i(r\cdot))\\ &\gtrsim s^k \sum_{i=0}^k H(p^{(x)}(r\cdot)) \gtrsim s^k \left\|\sum_{i=0}^k p^{(x)}(r\cdot)\right\|_{L^\infty(B(0,1))} \gtrsim s^k \|\tilde p\|_{L^\infty(B(0,r))},\end{split}\end{equation} where the implicit constants depend on only $n$ and $k$. It immediately follows that $$\widehat{\zeta}_k(p,x,sr) = \max_{k<j\leq d} \frac{\left\|p^{(x)}_j\right\|_{L^\infty(B(0,sr))}}{\left\|\tilde p\right\|_{L^\infty(B(0,sr))}} \lesssim \max_{k<j\leq d} s^{j-k}\,\frac{\left\|p^{(x)}_j\right\|_{L^\infty(B(0,r))}}{\left\|\tilde p\right\|_{L^\infty(B(0,r))}}\lesssim s\, \widehat{\zeta}_k(p,x,r),
$$ where the implied constant depends only on $n$ and $k$, and therefore, may be chosen to only depend on $n$ and $d$. The other inequality follows similarly and is left to the reader.
\end{proof}

We end with a statement about the joint continuity of $\widehat{\zeta}_k(p,x,r)$. Lemma \ref{zetacts} follow from elementary considerations; for some sample details, the reader may consult the proof of an analogous statement for $\zeta_k(p,x,r)$ in \cite[Lemma 2.8]{badgerflatpoints}.

\begin{definition} \label{d:coeff} A sequence of polynomials $(p^i)_{i=1}^\infty$ in $\RR^n$ converges \emph{in coefficients} to a polynomial $p$ in $\RR^n$ if $d=\max_i \deg p^i<\infty$ and $H(p-p^i)\rightarrow 0$ as $i\rightarrow\infty$.\end{definition}

\begin{lemma}\label{zetacts} For every $k\geq 0$, $\widehat{\zeta}_k(p,x,r)$ is jointly continuous in $p$, $x$, and $r$. That is, \begin{equation*} \widehat{\zeta}_k(p^i,x_i,r_i)\rightarrow \widehat{\zeta}_k(p,x,r)\end{equation*} whenever $\deg p>k$, $p^i\rightarrow p$ in coefficients, $x_i\rightarrow x\in\RR^n$, and $r_i\rightarrow r\in(0,\infty)$.\end{lemma}

\section{Growth estimates for harmonic polynomials}
\label{sect:growth}

We need several estimates on the growth of nonconstant harmonic polynomials of degree at most $k$. The main result of this section is the following uniform \L ojasiewicz inequality for harmonic polynomials of bounded degree.

\begin{theorem}[\L ojasiewicz inequality for harmonic polynomials]\label{lojasiewiczinequality}
For all $n\geq 2$ and $k \geq 1$, there exists a constant $c=c(n,k)> 0$  with the following property. If $p: \RR^n \rightarrow \RR$ is a nonconstant harmonic polynomial of degree at most $k$ and $x_0\in\Sigma_p$, then 
\begin{equation}\label{eqn:lojasiewiczinequality}
|p(z)| \geq c\|p\|_{L^\infty(B(x_0,1))}\mathrm{dist}(z,\Sigma_p)^k\quad\text{for all }z\in B(x_0,1/2).
\end{equation}
\end{theorem}

\begin{remark}\L ojasiewicz \cite{lojasiewicz} proved the remarkable result that if $f$ is a real analytic function on $\RR^n$ and $x_0\in \Sigma_f$ (the zero set of $f$), then there exist constants $C,\varepsilon,m>0$ such that $$|f(z)|\geq C\dist(z,\Sigma_f)^m\quad\text{for all }z\in B(x_0,\varepsilon).$$ The smallest possible $m$ is called the \emph{\L ojasiewicz exponent} of $f$ at $x_0$. It is perhaps a surprising fact that the \L ojasiewicz exponent of a polynomial can exceed the degree of the polynomial. Bounding the \L ojasiewicz exponent from above is a difficult problem in algebraic geometric; see e.g.~\cite{kollar}, \cite{so'n}. The content of Theorem \ref{lojasiewiczinequality} over the general form of the \L ojasiewicz inequality is the tight bound on the \L ojasiewicz exponent and uniformity of the constant $c$ in \eqref{eqn:lojasiewiczinequality} across all harmonic polynomials of bounded degree.
\end{remark}

The key tools that we use in this section are Almgren's frequency formula and Harnack's inequality for positive harmonic functions. Let us now recall the definition of the former.

\begin{definition}\label{almgrenfrequency}
 Let $f \in H^1_{\mathrm{loc}}(\mathbb R^n)$ and let $x_0\in \Sigma_f=\{x\in\RR^n:f(x)=0\}$. For all $r>0$, define the quantities $H(r,x_0,f)$ and $D(r,x_0,f)$ by $$
 H(r,x_0,f) = \int_{\partial B(x_0,r)} f^2\,d\sigma\quad\text{and}\quad D(r,x_0,f) = \int_{B(x_0,r)} |\nabla f|^2\,dx.$$ Then the \emph{frequency function} $N(r,x_0,f)$ is defined by
 $$N(r,x_0,f) = \frac{rD(r,x_0,f)}{H(r,x_0,f)}\quad\text{for all }r>0.$$
\end{definition}

Almgren introduced the frequency function in \cite{almgren}. It is a simple matter to show that for any harmonic polynomial $p$, the frequency function $N(r,x_0,p)\equiv \deg p$. When $f$ is any harmonic function, not necessarily a polynomial, Almgren proved that $N(r, x_0, f)$ is absolutely continuous in $r$ and monotonically decreasing as $r\downarrow 0$, and moreover, $\lim_{r\downarrow 0} N(r,x_0,f)$ is the order to which $f$ vanishes at $x_0$. It can also be verified that \begin{equation}\label{logderofH}
\frac{d}{dr} \log\left(\frac{H(r, x_0, f)}{r^{n-1}}\right) = 2\frac{N(r, x_0, f)}{r}.
\end{equation}
Integrating \eqref{logderofH} and invoking the monotonicity of $N(r,x_0,f)$ in $r$, one can prove the following doubling property. For a proof of Lemma \ref{l:doublinglemma}, see e.g.~ \cite[Corollary 1.5]{hannotes}; the result is stated there with $x_0 = 0$ and $R =1$, but the general case readily follows by observing that $N(R, x_0, f) = N(1, 0, g)$, where $g(x)=f(x_0+Rx)/R$.

\begin{lemma} \label{l:doublinglemma}
If $f$ is a harmonic function on $B(x_0,R)$, then for all $r\in(0,R/2)$,
\begin{equation}\label{doublingequation}
\fint_{B(x_0,2r)} f^2\, dx \leq 2^{2N(R, x_0, f)\, -\, 1} \fint_{B(x_0,r)} f^2\,dx.
\end{equation}
\end{lemma}

\begin{corollary}\label{supcomparabletoaverage} For all $n\geq 2$ and $k\geq 1$, there exists a constant $C>0$ such that if $p:\RR^n\rightarrow\RR$ is a harmonic polynomial of degree at most $k$, $x_0\in\RR^n$, and $r>0$, then \begin{equation}\label{eq:supcomparabletoaverage} \fint_{B(x_0,2r)} p^2\,dx\leq C \fint_{B(x_0,r)} p^2\,dx\quad\text{and}\quad\sup_{B(x_0, r)}p^2 \leq 2^n C \fint_{B(x_0,r)} p^2\,dx.\end{equation}
\end{corollary}

\begin{proof} The first inequality in \eqref{eq:supcomparabletoaverage} is an immediate consequence of Lemma \ref{l:doublinglemma} and the well-known fact that $N(r,x_0,p)\equiv \deg p$ for every harmonic polynomial $p$.

To establish the second inequality in \eqref{eq:supcomparabletoaverage}, first note that $B(z, r) \subseteq B(x_0, 2r)$ for all $z \in B(x_0,r)$. By the mean value property of harmonic functions and the first inequality, \begin{equation*} p(z)^2 = \left(\fint_{B(z, r)} p\,dx\right)^2 \leq \fint_{B(z,r)} p^2\, dx \leq 2^n \fint_{B(x_0, 2r)} p^2\,dx \leq 2^n C \fint_{B(x_0, r)} p^2\,dx.\end{equation*} This establishes \eqref{eq:supcomparabletoaverage}.\end{proof}

Next, as an application of Corollary \ref{supcomparabletoaverage} and Harnack's inequality, we show that $p(z)$ is relatively large when $z$ is far enough away from $\Sigma_p$.

\begin{lemma}\label{nontangentialestimate} For all $n\geq 2$ and $k\geq 1$, there exists a constant $c>0$ such that if $p:\RR^n\rightarrow\RR$ is a harmonic polynomial of degree at most $k$, $z\in \RR^n$, and $x_0\in\Sigma_p$ is any point such that $\rho:=\dist(z,\Sigma_p)=|z-x_0|$, then   \begin{equation}\label{e:nte} |p(z)| \geq c\sup_{B(x_0,\rho)} |p|.\end{equation}
\end{lemma}
\begin{proof} Let $n\geq 2$ and $k\geq 1$ be given, and let $p:\RR^n\rightarrow\RR$ be a harmonic polynomial of degree at most $k$. Since the conclusion is trivial for all $z\in\Sigma_p$, we may assume $z\in\RR^n\setminus\Sigma_p$.
Without loss of generality, we may further assume that $p$ is positive in $B(z, \rho)$, where $\rho=\dist(z,\Sigma_p)$. By Harnack's inequality for positive harmonic functions (e.g., see \cite[Theorem 3.4]{HFT}), there exists a constant $A=A(n) > 0$ such that $$p(z)^2 \geq A \sup_{B(z, \rho/2)} p^2\geq A \fint_{B(z,\rho/2)} p^2\,dx.$$ Pick $x_0\in\Sigma_p$ such that $\rho=|z-x_0|$ and note that $B(z,2\rho)\supseteq B(x_0,\rho)$. Hence, by two applications of the first inequality in Corollary \ref{supcomparabletoaverage} and then by the second inequality, $$\fint_{B(z,\rho/2)} p^2\,dx \geq C^2 \fint_{B(z,2\rho)} p^2\,dx \geq 2^{-n}C^2\fint_{B(x_0,\rho)} p^2\,dx \geq 4^{-n} C \sup_{B(x_0,\rho)} p^2.$$ Combining the displayed equations, we conclude that \eqref{e:nte} holds with $c= 2^{-n}\sqrt{AC}$.
\end{proof}

We can now obtain the \L ojasiewicz inequality for harmonic polynomials (Theorem \ref{lojasiewiczinequality}) by combining Lemma \ref{nontangentialestimate} with the estimate \eqref{e:polyshrink} from the proof of Lemma \ref{l:quasilinear}.

\begin{proof}[Proof of Theorem \ref{lojasiewiczinequality}]Let $n\geq 2$ and $k\geq 1$ be given. Suppose that $p:\RR^n\rightarrow\RR$ is a nonconstant harmonic polynomial of degree at most $k$, and without loss of generality, assume that $0\in\Sigma_p$ (the origin will play the role of $x_0$ in the statement of the theorem). Fix $z\in B(0,1/2)$ and choose $x_0\in\Sigma_p$ to be any point such that $\rho:=|z-x_0|=\dist(z,\Sigma_p)$. Note that $\rho< 1/2$, since $0\in\Sigma_p$ and $z\in B(0,1/2)$.
On one hand, by Lemma \ref{nontangentialestimate}, $$|p(z)| \gtrsim \sup_{B(x_0, \rho)} |p|.$$ On the other hand, applying \eqref{e:polyshrink} with $r=2$ and $s=\rho/2$ (this is fine as $s<1$), $$\sup_{B(x_0, \rho)} |p| \gtrsim \rho^k\sup_{B(x_0, 2)} |p|\geq \rho^k \|p\|_{L^\infty(B(0,1))}.$$ Here all implicit constants depend on at most $n$ and $k$. The inequality \eqref{eqn:lojasiewiczinequality} immediately follows by combining the displayed equations (and recalling the definition of $\rho$).
\end{proof}

As we work separately with the sets $\{p > 0\}$ and $\{p < 0\}$ below, it is important for us to know that $\sup p^+$ and $\sup p^-$ are comparable in any ball centered on $\Sigma_p$.

\begin{lemma}\label{polynomialsarebalanced} For all $n\geq 2$ and $k\geq 1$, there exists a constant $C>1$ such that if $p:\RR^n\rightarrow\RR$ is a nonconstant harmonic polynomial of degree at most $k$, then
\begin{equation}\label{eq:polynomialsarebalanced} C^{-1} \sup_{B(x_0,r)} p^+ \leq \sup_{B(x_0,r)} p^- \leq C\sup_{B(x_0,r)} p^+\quad\text{for all $x_0\in\Sigma_p$ and $r>0$}. \end{equation}
\end{lemma}

\begin{proof} Let $M^\pm = \sup_{B(x_0,r)} p^\pm$, and assume without loss of generality that $M^+\geq M^-$. The argument now splits into two cases.

Case I. Assume that $\sup_{B(x_0,r/2)} |p| = \sup_{B(x_0,r/2)} p^-$. Then by the estimate \eqref{e:polyshrink} in the proof of Lemma \ref{l:quasilinear}, $$M^- \geq \sup_{B(x_0,r/2)} p^- = \sup_{B(x_0,r/2)} |p| \gtrsim \sup_{B(x_0,r)}|p|=M^+,$$ where the implicit constant depends only on $n$ and $k$.

Case II. Assume that $\sup_{B(x_0,r/2)} |p| = \sup_{B(x_0,r/2)} p^+$. Note that $p + 2M^-$ is a positive harmonic function in $B(x_0,r)$. Thus, by Harnack's inequality, \begin{equation}\label{compareMminus}2M^- = p(x_0)+2M^{-} \geq a \sup_{B(x_0, r/2)} (p + 2M^-)= a \sup_{B(x_0,r/2)} (p^+ +2M^{-}),\end{equation} where $a=a(n)>0$. We now argue as in Case I. By \eqref{e:polyshrink}, $$\sup_{B(x_0,r/2)} p^+ = \sup_{B(x_0,r/2)} |p| \gtrsim \sup_{B(x_0,r)}|p|=M^+,$$ where the implicit constant depends only on $n$ and $k$. Combining the displayed equations, we conclude that $M^- \gtrsim M^+$.
\end{proof}

Finally, we record a technical observation that will be needed in \S6.

\begin{lemma}\label{l:L2L2} Let $n\geq 2$ and let $k\geq 1$. If $p:\RR^n\rightarrow\RR$ is a harmonic polynomial of degree at most $k$, then $\|p\|_{L^2(B(0,1))} \sim_{n,k} \|p\|_{L^2(\partial B(0,1))}$.\end{lemma}

\begin{proof} The fact that $\|p\|_{L^2(\partial B(0,1))}$ is a norm on the space of harmonic polynomials follows from the maximum principle for harmonic functions. Thus, the equivalence of $\|p\|_{L^2(B(0,1))}$ and $\|p\|_{L^2(\partial B(0,1))}$ for harmonic polynomials of bounded degree follows from the equivalence of norms on finite-dimensional vector spaces. \end{proof}

\section{$\cH_{n,k}$ points are detectable in $\cH_{n,d}$}
\label{sect:detectable}

The next lemma shows that $\hz_k$ (see Definition \ref{modifiedzeta} above) controls how close $\Sigma_p\in\cH_{n,d}$ is to the zero set of a harmonic polynomial of degree at most $k$; cf.~\cite[Lemma 4.1]{badgerflatpoints}. For the definition of the bilateral approximation number $\Theta_{\Sigma_p}^{\cH_{n,k}}(x,r)$, we refer the reader to the introduction (see \eqref{bil-app-num}).

\begin{lemma}\label{thetacontrolledbyzeta} For all $n\geq 2$ and $d\geq 2$, there exists $0< C<\infty$ such that for every harmonic polynomial $p:\RR^n\rightarrow\RR$ of degree $d$ and for every $1\leq k< d$, \begin{equation}\label{e:theta} \Theta^{\cH_{n,k}}_{\Sigma_p}(x,r) \leq C\,\widehat{\zeta}_k(p,x,r)^{1/k}\quad\text{for all } x\in\Sigma_p\text{ and }r>0.\end{equation}
\end{lemma}

\begin{proof} Let $p:\RR^n\rightarrow\RR$ be a harmonic polynomial of degree $d\geq 2$, let $1\leq k<d$, and let $x\in\Sigma_p$. Write $p(\cdot+x)=p^{(x)}_d+\dots+p^{(x)}_{k+1}+p^{(x)}_k+\dots+p^{(x)}_1$, where each $p^{(x)}_i:\RR^n\rightarrow\RR$ is an $i$-homogeneous polynomial in $y$ with coefficients depending on $x$. We remark that $x+\Sigma_{p(\cdot+x)}=\Sigma_p$. Now, since $p$ is harmonic, each term $p_i^{(x)}$ is harmonic, as well. Set $\tilde p= p^{(x)}_k+\dots+p^{(x)}_1$, the low order part of $p$ at $x$, and note that $\tilde p(0)=0$. If $\tilde p\equiv 0$, then $\widehat{\zeta}_k(p,x,r)=\infty$ for all $r>0$ and \eqref{e:theta} holds trivially. Thus, we may assume that $\tilde p\not\equiv 0$, in which case $\Sigma_{\tilde p}\in \cH_{n,k}$. To prove (\ref{e:theta}), we shall prove a slightly stronger pair of inequalities, \begin{equation}\label{e:theta-1} r^{-1}\sup_{a\in\Sigma_p\cap B(x,r)}\dist(a,(x+\Sigma_{\tilde p})\cap \overline{B(x,r)}) \leq C_1\,\widehat{\zeta}_k(p,x,r)^{1/k}
\end{equation} and \begin{equation}\label{e:theta-2}r^{-1}\sup_{w\in (x+\Sigma_{\tilde p})\cap B(x,r)}\dist(w,\Sigma_p)\leq C_2\,\widehat{\zeta}_k(p,x,2r)^{1/k} \end{equation} for some constants $C_1$ and $C_2$ that depend only on $n$, $d$, and $k$, and therefore, may be chosen to depend only on $n$ and $d$. With the help of Lemma \ref{l:quasilinear}, \eqref{e:theta} follows immediately from \eqref{e:theta-1} and \eqref{e:theta-2}.

Suppose $\tilde p(z)\neq 0$ for some $z\in B(0,r)$ and choose $y\in\Sigma_{\tilde p}\cap \overline{B(0,r)}$ such that $\rho:=\dist(z,\Sigma_{\tilde p}\cap \overline{B(0,r)})=|z-y|$. We note that $\rho\leq r$, since $\tilde p(0)=0$, and $B(0,r)\subseteq B(y,2r)$. Hence, by Lemma \ref{nontangentialestimate}, \begin{equation*}\label{eq:compareinteriortoglobe}
|\tilde{p}(z)| \geq c\|\tilde{p}\|_{L^\infty(B(y,\rho))} \stackrel{\eqref{e:polyshrink}}{\geq} c\left(\frac{\rho}{2r}\right)^k \|\tilde{p}\|_{L^\infty(B(y,2r))}\geq c\left(\frac{\rho}{r}\right)^k\|\tilde{p}\|_{L^\infty(B(0,r))},\end{equation*} where at each occurrence $c$ denotes a positive constant determined by $n$ and $k$. Thus, \begin{align*}
|p(z+x)| &\geq |\tilde p(z)| - \sum_{j=k+1}^d \|p_j^{(x)}\|_{L^\infty(B(0,r))}\\ & \geq c_1\left(\frac{\rho}{r}\right)^k \|\tilde p\|_{L^\infty(B(0,r))} - (d-k)\widehat{\zeta}_k(p,x,r)\|\tilde p\|_{L^\infty(B(0,r))},
\end{align*} where $c_1>0$ is a constant depending only on $n$ and $k$. It follows that $|p(z+x)|>0$ whenever $z\in B(0,r)$ and $\dist(z,\Sigma_{\tilde p}\cap\overline{B(0,r)})=\rho>C_1 \widehat{\zeta}_k(p,x,r)^{1/k}r,$ where $$C_1=\left(\frac{d-k}{c_1}\right)^{1/k}.$$ Consequently, for any $a=z+x\in\Sigma_p\cap B(x,r)$, we have $$\dist(a,(x+\Sigma_{\tilde p})\cap\overline{B(x,r)})=\dist(z,\Sigma_{\tilde p}\cap \overline{B(0,r)}) \leq C_1\widehat{\zeta}_k(p,x,r)^{1/k}r.$$ This establishes \eqref{e:theta-1}.

Next, suppose that $w\in (x+\Sigma_{\tilde p})\cap B(x,r)$, say $w=x+z$ for some $z\in \Sigma_{\tilde p}\cap B(0,r)$. Let $\delta<r$ be a fixed scale, to be chosen below. Because $\tilde p$ is harmonic, we can locate points $z_\delta^\pm\in\partial B(z,\delta)$ such that $$\tilde p(z_\delta^+)=\max_{z'\in \overline{B(z,\delta)}}\tilde p(z')>0\quad\text{and}\quad \tilde p(z_\delta^-)=\min_{z'\in \overline{B(z,\delta)}}\tilde p(z')<0.$$ Thus, by Lemma \ref{polynomialsarebalanced}, $$\pm \tilde p(z^\pm_{\delta})=|\tilde p(z^\pm_{\delta})|\geq c\|\tilde p\|_{L^\infty(B(z,\delta))} \stackrel{\eqref{e:polyshrink}}{\geq} c\left(\frac{\delta}{3r}\right)^k\|\tilde p\|_{L^\infty(B(z,3r))}\geq c\left(\frac{\delta}{r}\right)^k\|\tilde p\|_{L^\infty(B(0,2r))},$$ where at each occurence $c>0$ depends only on $n$ and $k$. We conclude that \begin{align*} \pm p(z^\pm_\delta+x) &\geq \pm\tilde p(z^\pm_\delta)- \sum_{j=k+1}^d \|p_j^{(x)}\|_{L^\infty (B(0,2r))} \\
&\geq c_2\left(\frac{\delta}{r}\right)^k\|\tilde p\|_{L^\infty(B(0,2r))}-(d-k)\widehat{\zeta}_k(p,x,2r)\|\tilde p\|_{L^\infty(B(0,r))}>0\end{align*} provided that $\delta > C_2\widehat{\zeta}_k(p,x,2r)^{1/k}r$, where $C_2=\left[(d-k)/c_2\right]^{1/k}$. But we also required $\delta < r$ above. To continue, there are two cases. On one hand, if $C_2\widetilde{\zeta}_k(p,x,2r)^{1/k}\geq 1$, then $\Theta_{\Sigma_p}^{\cH_{n,k}}(x,r)\leq 1 \leq C_2\widetilde{\zeta}_k(p,x,2r)^{1/k}$ holds trivially. On the other hand, suppose that $C_2\widetilde{\zeta}_k(p,x,2r)^{1/k}<1$. In this case, pick any $\delta\in (C_2\widetilde{\zeta}_k(p,x,2r)^{1/k}r,r)$. Then the estimate above gives $\pm p(z^\pm_\delta+x)>0$. In particular, the straight line segment $\ell$ that connects $z^+_\delta+x$ to $z^-_{\delta}+x$ inside $\overline{B(z+x,\delta)}$ must intersect $\Sigma_p\cap \overline{B(z+x,\delta)}$ by the intermediate value theorem and the convexity of ball. Hence $\dist(w,\Sigma_p)=\dist(z+x,\Sigma_p)\leq \delta$. Therefore, letting $\delta\downarrow C_2\widetilde{\zeta}_k(p,x,2r)^{1/k}$, we obtain \eqref{e:theta-2}.\end{proof}

\begin{remark} In the proof of Lemma \ref{thetacontrolledbyzeta}, the harmonicity of $p$ was only used to establish the harmonicity of $\tilde p$. Thus, the argument actually yields that $\Theta_{\Sigma_p}^{\cH_{n,k}}(x,r)\lesssim_{n,d} \widehat{\zeta}_k(p,x,r)$ for all $x\in\Sigma_p$ and for all $r>0$, whenever $p:\RR^n\rightarrow\RR$ is a polynomial of degree $d>k$ such that $\tilde p=p^{(x)}_k+\dots+p^{(x)}_1$ is harmonic.\end{remark}

The following useful fact facilitates normal families arguments with sequences in $\cH_{n,d}$. It is ultimately a consequence of the mean value property of harmonic functions.

\begin{lemma}\label{convergenceofzerosets}
Suppose that $\Sigma_{p_1},\Sigma_{p_2},\dots \in \cH_{n,d}$. If $p_i \rightarrow p$ in coefficients and $H(p)\neq 0$, then $\Sigma_p \in \cH_{n,d}$ and $\Sigma_{p_i} \rightarrow \Sigma_p$ in the Attouch-Wets topology (see Appendix \ref{sect:2}).
\end{lemma}

\begin{proof} Suppose that, for each $i\geq 1$, $p_i:\RR^n\rightarrow\RR$ is a harmonic polynomial of degree at most $d$ such that $p_i(0)=0$. Assume that $p_i\rightarrow p$ in coefficients and $H(p)\neq 0$. Then $p:\RR^n\rightarrow\RR$ is also a harmonic polynomial of degree at most $d$ such that $p(0)=0$, because $p_i\rightarrow p$ uniformly on compact subsets of $\RR^n$, and $p$ is nonconstant, because $H(p)\neq 0$. Hence $\Sigma_p\in\cH_{n,d}$. It remains to show that $\Sigma_{p_i}\rightarrow \Sigma_p$ in the Attouch-Wets topology, which is metrizable. Thus, it suffices to prove that every subsequence $(\Sigma_{p_{ij}})_{j=1}^\infty$ of $(\Sigma_{p_i})_{i=1}^\infty$ has a further subsequence $(\Sigma_{p_{ijk}})_{k=1}^\infty$ such that $\Sigma_{p_{ijk}}\rightarrow \Sigma_p$ in the Attouch-Wets topology.

Fix an arbitrary subsequence $(\Sigma_{p_{ij}})_{j=1}^\infty$ of $(\Sigma_{p_i})_{i=1}^\infty$. Since $0\in\Sigma_{p_{ij}}$ for all $j\geq 1$ and the set of closed sets in $\RR^n$ containing the origin is sequentially compact, there exists a closed set $F\subseteq\RR^n$ containing $0$ and a subsequence $(\Sigma_{p_{ijk}})_{k=1}^\infty$ of $(\Sigma_{p_{ij}})_{j=1}^\infty$ such that $\Sigma_{p_{ijk}}\rightarrow F$. We claim that $F=\Sigma_p$. Indeed, on one hand, for any $y\in F$ there exists a sequence $y_k\in\Sigma_{p_{ijk}}$ such that $y_k\rightarrow y$; but $p(y)=\lim_{k\rightarrow\infty} p_{ijk}(y_k)=\lim_{k\rightarrow\infty} 0=0$, since $y_k\in\Sigma_{p_{ijk}}$, $p_{ijk}\rightarrow p$ uniformly on compact sets, and $y_k\rightarrow y$. Hence $y\in\Sigma_p$ for all $y\in F$. That is, $F\subseteq\Sigma_p$. On the other hand, suppose $z\in \Sigma_p$. Since $p(z)=0$, but $p\not\equiv 0$, for all $r\in(0,1)$ we can locate points $z^\pm_r\in B(z,r)$ such that $p(z^+_r)>0$ and $p(z^-_r)<0$ by the mean value theorem for harmonic functions. Because  $p_{ijk}\rightarrow p$ pointwise, it follows that $$p_{ijk}(z^+_r)>0\quad\text{and}\quad p_{ijk}(z^-_r)<0$$ for all sufficiently large $k$ depending on $r$. In particular, by the intermediate value theorem, the straight line segment connecting $z^+_r$ to $z^-_r$ inside $B(z,r)$ must intersect $\Sigma_{p_{ijk}}\cap B(z,r)$ for all sufficiently large $k$ depending on $r$. Hence $\dist(z,\Sigma_{p_{ijk}}\cap B(z,1))\rightarrow 0$ as $k\rightarrow\infty$. Ergo, since $\Sigma_{p_{ijk}}\rightarrow F$ in the Attouch-Wets topology, $$\dist(z,F) \leq
\liminf_{k\rightarrow\infty}\left(\dist(z,\Sigma_{p_{ijk}}\cap B(z,1))+\ex(\Sigma_{p_{ijk}}\cap B(z,1),F)\right)=0.$$ That is, $z\in F$ for all $z\in \Sigma_p$. Therefore, $\Sigma_p\subseteq F$, and the conclusion follows.
\end{proof}

\begin{corollary}\label{c:closed} For all $n\geq 2$ and $1\leq k\leq d$, $\cH_{n,d}$ and $\cF_{n,k}$ are closed subsets of $\CL{0}$ with the Attouch-Wets topology.\end{corollary}

\begin{proof} Suppose $\Sigma_{p_i}\in\cH_{n,d}$ for all $i\geq 1$ and $\Sigma_{p_i}\rightarrow F$ for some closed set $F$ in $\RR^n$. Replacing each $p_i$ by $p_i/H(p_i)$, which leaves $\Sigma_{p_i}$ unchanged, we may assume $H(p_i)=1$ for all $i\geq 1$. Hence we can find a polynomial $p$ and a subsequence $(p_{ij})_{j=1}^\infty$ of $(p_i)_{i=1}^\infty$ such that $p_{ij}\rightarrow p$ in coefficients and $H(p)=1$. Thus, by Lemma \ref{convergenceofzerosets}, $\Sigma_p\in\cH_{n,d}$ and $\Sigma_{p_{ij}}\rightarrow \Sigma_p$. Therefore, $F=\lim_{i\rightarrow\infty} \Sigma_{p_i}=\lim_{j\rightarrow\infty}\Sigma_{p_{ij}}=\Sigma_p\in \cH_{n,d}$. We conclude that $\cH_{n,d}$ is closed. Finally, $\cF_{n,k}$ is closed by the additional observation that $p$ is homogeneous of degree $k$ whenever $p_{ij}$ is homogeneous of degree $k$ for all $j$.\end{proof}

\begin{remark} \label{r:lac} For any $\Sigma_p\in\cH_{n,d}$ and $\lambda>0$, the dilate $\lambda \Sigma_p=\Sigma_q$, where $q:\RR^n\rightarrow\RR$ is given by $q(x)=p(x/\lambda)$ for all $x\in\RR^n$. Since $p$ is a nonconstant polynomial of degree at most $d$ such that $p(0)=0$, so is $q$. Also, $q$ is $k$-homogeneous, whenever $p$ is $k$-homogeneous. Finally, since $p$ is harmonic on $\RR^n$, the mean value theorem gives $$\fint_{B(y,r)}q(x)\ dx = \fint_{B(y,r)} p(x/\lambda)\ dx = \fint_{B(y/\lambda, r/\lambda)} p(x)\ dx = p(y/\lambda) = q(y)$$ for all $y\in \mathbb R^n$ and $r > 0$. Thus, since $q$ is continuous, it is also harmonic by the mean value theorem. This shows that $\lambda\Sigma_p\in\cH_{n,d}$ for all $\Sigma_p\in\cH_{n,d}$ and $\lambda>0$. Likewise, $\lambda\Sigma_p\in\cF_{n,k}$ for all $\Sigma_p\in\cF_{n,k}$ and $\lambda>0$. In other words, $\cH_{n,d}$ and $\cF_{n,k}$ are cones. Therefore, $\cH_{n,d}$ and $\cF_{n,k}$ are local approximation classes in the sense of Definition \ref{d:big}(i). A similar argument shows that $\cH_{n,d}$ is translation invariant in the sense that $\Sigma_p-x\in \cH_{n,d}$ for all $\Sigma_p\in\cH_{n,d}$ and $x\in\Sigma_p$.
\end{remark}

The next lemma captures a weak rigidity property of real-valued harmonic functions: the zero set of a real-valued harmonic function determines the relative arrangement of its positive and negative components.

\begin{lemma} \label{l:2side} Let $f:\RR^n\rightarrow\RR$ and $g:\RR^n\rightarrow\RR$ be harmonic functions, and let $\Sigma_f$ and $\Sigma_g$ denote the zero sets of $f$ and $g$, respectively. If $\Sigma_f=\Sigma_g$, then $f$ and $g$ take the same or the opposite sign simultaneously on every connected component of $\RR^n\setminus\Sigma_f=\RR^n\setminus\Sigma_g$. \end{lemma}
\begin{proof} Since the conclusion is trivial if $f$ is identically zero, we may assume in addition to the hypothesis that $f$ is not identically zero. According to \cite[Theorem 1.1]{lm-ratios}, if $u$ and $v$ are harmonic functions defined on a domain $\Omega\subseteq\RR^n$ whose zero sets satisfy $\Sigma_v\subseteq \Sigma_u$, then there exists a real-analytic function $\alpha$ in $\Omega$ such that $u=\alpha v$. Invoking this fact twice, we obtain that $f=\alpha g=\alpha\beta f$, where $\alpha$ and $\beta$ are real analytic functions on $\RR^n$. Since $f$ is not identically zero, it follows that $\alpha\beta=1$ on $\RR^n$. In particular, $\sign(\alpha)=\pm1$ on $\RR^n$. Therefore, $\sign(f)=\sign(\alpha)\sign(g)=\pm\sign(g)$ on $\RR^n$.
\end{proof}

The following lemma indicates that zero sets of homogeneous harmonic polynomials of different degrees are uniformly separated on balls centered at the origin. This answers affirmatively a question posed in \cite[Remark 4.12]{badgerflatpoints}.

\begin{lemma}\label{zerosetsarefarapart} For all $n\geq 2$ and $1\leq j<k$, there exists a constant $\varepsilon>0$ such that for all $\Sigma_p\in\cF_{n,k}$ and $\Sigma_q\in\cF_{n,j}$,
$$\mD{0}{r}[\Sigma_p,\Sigma_q] = \frac{1}{r}\max\left\{\sup_{x\in \Sigma_p\cap B(0,r)}\dist(x,\Sigma_q), \sup_{y\in\Sigma_q\cap B(0,r)} \dist(y,\Sigma_p)\right\} \geq \varepsilon\quad\text{for all }r>0.$$
\end{lemma}

\begin{proof} Note that $\lambda\Sigma_p=\Sigma_p$ and  $\lambda\Sigma_q=\Sigma_q$ for all $\lambda>0$ whenever $\Sigma_p\in\cF_{n,k}$ and $\Sigma_q\in\cF_{n,j}$. Hence $\mD{0}{r}[\Sigma_p,\Sigma_q]=\mD{0}{1}[r^{-1}\Sigma_p,r^{-1}\Sigma_q]=\mD{0}{1}[\Sigma_p,\Sigma_q]$ for all $r>0$, whenever $n\geq 2$, $1\leq j<k$, $\Sigma_p\in\cF_{n,k}$, and $\Sigma_q\in\cF_{n,j}$. Thus, it suffices to prove the claim with $r=1$.

Assume to the contrary that for some $n\geq 2$ and $1\leq j<k$ we can find sequences $p_1,p_2,\dots \in \cF_{n,k}$ and $q_1,q_2,\dots\in\cF_{n,j}$ such that  \begin{equation}\label{e:pi-qi}\mD{0}{1}[\Sigma_{p_i},\Sigma_{q_i}]\leq \frac{1}{i}\quad\text{for all }i\geq 1.\end{equation}
By Corollary \ref{c:closed}, passing to subsequences (which we relabel), we may assume that there exist $\Sigma_p\in \cF_{n,k}$ and $\Sigma_q\in\cF_{n,j}$ such that $\Sigma_{p_i}\rightarrow \Sigma_p$ and $\Sigma_{q_i}\rightarrow\Sigma_q$. Moreover, replacing each $p_i$ and $q_i$ by $p_i/H(p_i)$ and $q_i/H(q_i)$, respectively, and passing to further subsequences (which we again relabel), we may assume that $p_i\rightarrow p$ in coefficients and $q_i\rightarrow q$ in coefficients, where $p$ and $q$ are homogeneous harmonic polynomials of degree $k$ and $j$, respectively. By two applications of the weak quasitriangle inequality (see Appendix \ref{sect:2}),
\begin{align}\label{quasi-triang-ineq}
 \mD{0}{1/4}[\Sigma_p,\Sigma_q] &\leq 2\mD{0}{1/2}[\Sigma_p,\Sigma_{p_i}] + 2\mD{0}{1/2}[\Sigma_{p_i},\Sigma_q]
\\ &\leq 2\mD{0}{1/2}[\Sigma_p,\Sigma_{p_i}]+4\mD{0}{1}[\Sigma_{p_i},\Sigma_{q_i}] + 4 \mD{0}{1}[\Sigma_{q_i},\Sigma_q].\nonumber
 \end{align}
 Letting $i\rightarrow\infty$, we have the first term vanishes since $\Sigma_{p_i}\rightarrow\Sigma_p$, the second term vanishes by \eqref{e:pi-qi}, and the the third term vanishes since $\Sigma_{q_i}\rightarrow \Sigma_q$. Hence $\mD{0}{1/4}[\Sigma_{p},\Sigma_{q}]=0$, which implies $\Sigma_p\cap B(0,1/4)=\Sigma_q\cap B(0,1/4)$. But $\Sigma_p$ and $\Sigma_q$ are cones, so in fact $\Sigma_p=\Sigma_q$.
By Lemma \ref{l:2side}, the functions $p$ and $q$ take the same or the opposite sign simultaneously on every connected component of $\RR^n\setminus\Sigma_p = \RR^n\setminus\Sigma_q$. Hence either $p(x)q(x)\geq 0$ for all $x\in\RR^n$ or $p(x)q(x)\leq 0$ for all $x\in\RR^n$. It follows that either $\int_{S^{n-1}} pq\,d\sigma>0$ or $\int_{S^{n-1}} pq\,d\sigma <0$. This contradicts the fact that homogeneous harmonic polynomials of different degrees are orthogonal in $L^2(S^{n-1})$ (e.g.~see \cite[Proposition 5.9]{HFT}).
\end{proof}

We now show that $\hz_k$ cannot grow arbitrarily large as $\Theta^{\cH_{n,k}}_{\Sigma_p}$ becomes arbitrarily small; cf.~ \cite[Proposition 4.8]{badgerflatpoints}.

\begin{lemma}\label{ifthetaissmallthensoiszeta} For all $n\geq 2$ and $1\leq k<d$ there is $\delta_{n,d,k}>0$ with the following property. If $p:\RR^n\rightarrow\RR$ is a harmonic polynomial of degree $d$ and $\Theta^{\cH_{n,k}}_{\Sigma_p}(x,r)<\delta_{n,d,k}$ for some $x\in\Sigma_p$ and $r>0$, then $\widehat{\zeta}_k(p,x,r)<\delta_{n,d,k}^{-1}$.
\end{lemma}

\begin{proof}

Let $n\geq 2$ and $1\leq k<d$ be given. Suppose in order to reach a contradiction that for all $j\geq 1$ there exists a harmonic polynomial $p_j:\RR^n\rightarrow\RR$ of degree $d$, $x_j\in\Sigma_{p_j}$, and $r_j>0$ such that $\Theta^{\cH_{n,k}}_{\Sigma_{p_j}}(x_j,r_j)<1/j$, but $\widehat{\zeta}_k(p_j,x_j,r_j)\geq j$. Replacing each $p_j$ with $\tilde p_j$, $$\tilde p_j(y) = H(p_j)^{-1}\cdot p(r_j(y+x_j))\quad\text{for all }y\in\RR^n,$$ that is, left translating by $x_j$, dilating by $1/r_j$, and scaling by $1/H(p_j)$, we may assume without loss of generality that $x_j=0$, $r_j=1$, and $H(p_j)=1$ for all $j\geq 1$. Therefore, there exists a sequence $(p_j)_{j=1}^\infty$ of harmonic polynomials in $\RR^n$ of degree $d$ and height 1 with $p_j(0)=0$ such that $\Theta_{\Sigma_{p_j}}^{\cH_{n,k}}(0,1)\leq 1/j$, and $\widehat{\zeta}_k(p_j,0,1)\geq j$. Passing to a subsequence, we may assume that $p_j\rightarrow p$ in coefficients to some harmonic polynomial $p:\RR^n\rightarrow\RR$ with height 1. By Lemma \ref{convergenceofzerosets}, $\Sigma_{p_j}\rightarrow \Sigma_p$, as well. On one hand, \begin{equation}\label{e:small-1}\Theta_{\Sigma_p}^{\cH_{n,k}}(0,1/2) \leq 2 \liminf_{j\rightarrow\infty} \Theta_{\Sigma_{p_j}}^{\cH_{n,k}}(0,1)=0.\end{equation} (For a primer on the interaction of limits and approximation numbers, see Appendix \ref{sect:2}.) On the other hand, by Lemma \ref{l:height} and the fact that $\widehat{\zeta}_k(p_j,0,1)\geq j$, it must be that the height of the polynomial $p_j$ is obtained from the coefficient of some term of $p_j$ of degree at least $k+1$, provided that $j$ is sufficiently large. In particular, we conclude that $p$ has degree at least $k+1$. Hence $\widehat{\zeta}_k(p,0,1)$ is well defined and $\widehat{\zeta}_k(p,0,1)=\lim_{j\rightarrow\infty}\widehat{\zeta}_k(p_j,0,1)=\infty$ by Lemma \ref{zetacts}. Thus, the low order part of $p$ at $0$ (that is the terms of degree at most $k$) vanishes and $p$ has the form \begin{equation}\label{e:small-2} p=p^{(0)}_{d}+p^{(0)}_{d-1}+\dots +\dots+p^{(0)}_{i},\quad p^{(0)}_{i}\neq 0\text{ for some }i\geq k+1.\end{equation} We shall now show that \eqref{e:small-1} and \eqref{e:small-2} are incompatible with Lemma \ref{zerosetsarefarapart}:

By \eqref{e:small-1}, there exists $\Sigma_q\in\cl{\cH_{n,k}}=\cH_{n,k}$ such that $\Sigma_p\cap B(0,1/2)=\Sigma_q\cap B(0,1/2)$, say \begin{equation}\label{e:small-3} q=q^{(0)}_k+q^{(0)}_{k-1}+\dots+q^{(0)}_l,\quad q^{(0)}_l\neq 0\text{ for some }1\leq l\leq k.\end{equation} Choose any sequence $r_m\downarrow 0$ as $m\rightarrow\infty$.  By \eqref{e:small-2}, $r_m^{-i} p(r_m\cdot)\rightarrow p^{(0)}_i$ in coefficients and
by \eqref{e:small-3},
$r_m^{-l} q(r_m\cdot)\rightarrow q^{(0)}_l$ in coefficients also. Hence $r_m^{-1}\Sigma_p = \Sigma_{r_m^{-i}p(r_m\cdot)}\rightarrow \Sigma_{p^{(0)}_i}\in \cF_{n,i}$ and $r_m^{-1}\Sigma_q = \Sigma_{r_m^{-l}p(r_m\cdot)}\rightarrow \Sigma_{q^{(0)}_l}\in \cF_{n,l}$ by Lemma \ref{convergenceofzerosets}. By the weak quasitriangle inequality (applied twice as in \eqref{quasi-triang-ineq}), $$\mD{0}{1}\left[\Sigma_{p_i^{(0)}},\Sigma_{q_i^{(0)}}\right] \leq 2 \mD{0}{2}\left[\Sigma_{p_i^{(0)}},r_m^{-1}\Sigma_p\right]+4\mD{0}{4}\left[r_m^{-1}\Sigma_p,r_m^{-1}\Sigma_q\right]+4\mD{0}{4}\left[r_m^{-1}\Sigma_q,\Sigma_{q_l^{(0)}}\right].$$ As $m\rightarrow\infty$, the first and the last term vanish, because $r_m^{-1}\Sigma_p \rightarrow \Sigma_{p^{(0)}_i}$ and $r_m^{-1}\Sigma_q \rightarrow \Sigma_{q^{(0)}_l}$, respectively. Thus, \begin{equation*}\mD{0}{1}\left[\Sigma_{p_i^{(0)}},\Sigma_{q_l^{(0)}}\right] \leq \liminf_{m\rightarrow\infty} 4\mD{0}{4}\left[r_m^{-1}\Sigma_p,r_m^{-1}\Sigma_q\right] = \liminf_{m\rightarrow\infty} 4\mD{0}{4r_m}[\Sigma_p,\Sigma_q]=0,\end{equation*} where the ultimate equality holds because $\Sigma_p\cap B(0,1/2)=\Sigma_q\cap B(0,1/2)$ and $4r_m\downarrow 0$. But $\mD{0}{1}\left[\Sigma_{p_i^{(0)}},\Sigma_{q_l^{(0)}}\right]>0$ by Lemma \ref{zerosetsarefarapart}, because $\Sigma_{p_i^{(0)}}\in\cF_{n,i}$, $\Sigma_{q_l^{(0)}}\in \cF_{n,l}$, and $i>l$. We have reached a contradiction. Therefore, for all $n\geq 2$ and $1\leq k<d$, there exists $j\geq 1$ such that if $p:\RR^n\rightarrow\RR$ is a harmonic polynomial of degree $d$ and $\Theta_{\Sigma_p}^{\cH_{n,k}}(x,r)<1/j$ for some $x\in\Sigma_p$ and $r>0$, then $\widehat{\zeta}_k(p,x,r)<j$.\end{proof}

We now have all the ingredients required to prove Theorem \ref{t:main2}.

\begin{proof}[Proof of Theorem \ref{t:main2}] Given $n\geq 2$ and $1\leq k<d$, let $\delta_{n,d,k}>0$ denote the constant from Lemma \ref{ifthetaissmallthensoiszeta}. Let $p:\RR^n\rightarrow\RR$ be a harmonic polynomial of degree $d$ and let $x\in\Sigma_p$. Write $\tilde p= p^{(x)}_k+ \dots + p^{(x)}_1$ for the part of $p$ of terms of degree at most $k$, so that $\partial^\alpha p(x)\neq 0$ for some $|\alpha|\leq k$ if and only if $\tilde p \not\equiv 0$. On one hand, if $\tilde p\not\equiv 0$, then $\widehat{\zeta}_k(p,x,1)<\infty$, whence $$\Theta_{\Sigma_p}^{\cH_{n,k}}(x,r) \lesssim_{n,d} \widehat{\zeta}_k(p,x,r)^{1/k} \lesssim_{n,d} r^{1/k}\widehat{\zeta}_k(p,x,1)^{1/k}\rightarrow 0\quad\text{as }r\rightarrow 0$$ by Lemma \ref{thetacontrolledbyzeta} and Lemma \ref{l:quasilinear}. In particular, if $\tilde p\not\equiv 0$, then $\Theta_{\Sigma_p}^{\cH_{n,k}}(x,r)<\delta_{n,d,k}$ for some $r>0$. On the other hand, if $\Theta_{\Sigma_p}^{\cH_{n,k}}(x,r)<\delta_{n,d,k}$ for some $r>0$, then \begin{equation}\label{e:zeta-delta}\widehat{\zeta}_k(p,x,r)<\delta_{n,d,k}^{-1}<\infty\end{equation} by Lemma \ref{ifthetaissmallthensoiszeta}, whence $\tilde p\not\equiv 0$. Moreover, in this case, \begin{equation*}\Theta_{\Sigma_p}^{\cH_{n,k}}(x,sr) \lesssim_{n,d} \widehat{\zeta}_k(p,x,sr)^{1/k} \lesssim_{n,d} s^{1/k}\widehat{\zeta}_k(p,x,r)^{1/k} \lesssim_{n,d,k} s^{1/k}\quad\text{for all }s\in(0,1)\end{equation*} by Lemma \ref{thetacontrolledbyzeta}, Lemma \ref{l:quasilinear}, and \eqref{e:zeta-delta}.\end{proof}

 \begin{proof}[Proof of Corollary \ref{c:main2}] From \eqref{e:detect} in Theorem \ref{t:main2}, it immediately follows that $\cH_{n,k}$ points are $(\phi,\Phi)$ detectable in $\cH_{n,d}$ for $\phi=\min\{\delta_{n,k+1,k},\dots,\delta_{n,d,k}\}>0$ and some function $\Phi$ of the form $\Phi(s)=Cs^{1/k}$ for all $s\in(0,1)$ (see Definition \ref{d:Tpdp}).
\end{proof}

\section{Structure of sets locally bilaterally well approximated by $\cH_{n,d}$}

Now that we know $\cH_{n,k}$ points are detectable in $\cH_{n,d}$, we may obtain Theorem \ref{t:main1} from repeated use of Theorem \ref{t:open}.

\begin{proof}[Proof of Theorem \ref{t:main1}]
Let $n\geq 2$ and $d\geq 2$ be given. By Remark \ref{r:lac} and Corollary \ref{c:closed}, $\cH_{n,k}$ and $\cF_{n,k}$ are closed local approximation classes and $\cH_{n,k}$ is also translation invariant for all $k\geq 1$. Thus, we may freely make use the technology in \S\S A.3--A.5 of the appendix. Using Definition \ref{def-perp}, Theorem \ref{t:main2} yields
 $$\cH_{n,k}\cap \cH_{n,k-1}^\perp = \{\Sigma_p\in\cH_{n,k}:\liminf_{r\downarrow 0}\Theta^{\cH_{n,k-1}}_{\Sigma_p}(0,r)>0\}=\cF_{n,k}\quad\text{for all }k\geq 2.$$

Suppose that $A\subseteq\RR^n$ is locally bilaterally well approximated by $\cH_{n,d}$ and put $U_d=A$. Since $\cH_{n,d-1}$ points are detectable in $\cH_{n,d}$ (by Corollary \ref{c:main2}) and $U_d$ is locally bilaterally well approximated by $\cH_{n,d}$, by Theorem \ref{t:open} we can write $$U_d=(U_d)_{\cH_{n,d-1}}\cup (U_d)_{\cH_{n,d-1}^\perp}=:U_{d-1}\cup A_d,$$ where $U_{d-1}$ and $A_d$ are disjoint, $U_{d-1}$ is relatively open in $U_d$, $U_{d-1}$ is locally bilaterally well approximated by $\cH_{n,d-1}$, and $U_d$ is locally bilaterally well approximated \emph{along} $A_d$ by $\cH_{n,d}\cap \cH_{n,d-1}^\perp=\cF_{n,d}$, that is, $\limsup_{r\downarrow 0}\sup_{x\in K}\Theta^{\cF_{n,d}}_{U_d}(x,r)=0$ for every compact set $K\subseteq A_d$. In particular, the latter property implies that every $x\in A_d$ is an $\cF_{n,d}$ point of $U_d$ by Theorem \ref{t:tangent-well}. Next, since $\cH_{n,d-2}$ points are detectable in $\cH_{n,d-1}$, we may repeat the argument, \emph{mutatis mutandis}, to write $$U_{d-1}=(U_{d-1})_{\cH_{n,d-2}}\cup (U_{d-1})_{\cH_{n,d-2}^\perp}=: U_{d-2}\cup A_{d-1},$$ where $U_{d-2}$ and $A_{d-1}$ are disjoint, $U_{d-2}$ is relatively open in $U_{d-1}$, $U_{d-2}$ is locally bilaterally well approximated by $\cH_{n,d-2}$, $U_{d-1}$ is locally bilaterally well approximated along $A_{d-1}$ by $\cF_{n,d-1}$, and every $x\in A_{d-1}$ is an $\cF_{n,d-1}$ point of $U_{d-1}$. In fact, since $U_{d-1}$ is relatively open in $U_d$, we have $U_{d-2}$ is relatively open in $U_d$, $U_d$ is locally bilaterally well approximated along $A_{d-1}$ by $\cF_{n,d-1}$, and every $x\in A_{d-1}$ is an $\cF_{n,d-1}$ point of $U_d$, as well. After a finite number of repetitions, this argument shows that $$A=U_d=U_{d-1}\cup A_d=\dots=U_1\cup A_2\cup\dots\cup A_d,$$ where the sets $U_1,A_2,\dots,A_d$ are pairwise disjoint, $U_1$ is relatively open in $A$, $U_1$ is locally bilaterally well approximated by $\cH_{n,1}$, $U_k=U_1\cup A_2\cup\dots\cup A_k$ is relatively open in $A$ for all $2\leq k\leq d$, $U_k$ is locally bilaterally well approximated by $\cH_{n,k}$ for all $2\leq k\leq d$, $A$ is locally bilaterally well approximated along $A_k$ by $\cF_{n,k}$ for all $2\leq k\leq d$, and every $x\in A_k$ is an $\cF_{n,k}$ point of $A$ for all $2\leq k\leq d$. Finally, assign $A_1=U_1$. Since $A_1$ relatively open in $A$, $A_1$ is locally bilaterally well approximated by $\cH_{n,1}$, and $\cH_{n,1}=\cF_{n,1}$, we conclude that every $x\in A_1$ is a $\cF_{n,1}$ point of $A$ by Theorem \ref{t:tangent-well}. This verifies (i)--(iv) of Theorem \ref{t:main1} and (v) follows immediately from (ii) and (iii).

Next, we want to prove that $A_1$ is relatively dense in $A$. Suppose that $x\in A\setminus A_1$, say $x\in A_k$ for some $k\geq 2$. To find points in $A_1$ nearby $x$, we will rely on the following fact: By Remark \ref{r:detectx}, since $\cH_{n,1}$ points are detectable in $\cH_{n,d}$, there exist $\alpha,\beta>0$ such that \begin{equation}\begin{split}\label{e:m1-0} &\hbox{if $\Theta^{\cH_{n,d}}_A(y,r')<\alpha$ for all $0<r'\leq r$}\\ &\qquad\qquad \hbox{and $\Theta_A^{\cH_{n,1}}(y,r)<\beta$ for some $y\in A$ and $r>0$, then $y\in A_1$}.\end{split}\end{equation} To proceed, since $x$ is an $\cF_{n,k}$ point of $A$ and $\cF_{n,k}$ is closed, we can find a homogeneous harmonic polynomial $p:\RR^n\rightarrow\RR$  and sequence of scales $r_i\downarrow 0$ such that $r_i^{-1}(\overline{A}-x)\rightarrow \Sigma_p$ in the Attouch-Wets topology ($\Sigma_p$ is a tangent set of $\overline{A}$ at $x$). Pick any $z\in \Sigma_p$ such that $|Dp|(z)\neq 0$. (That we can always find such a point is evident, because the singular set of a polynomial has dimension at most $n-2$, while $\dim \Sigma_p=n-1$.) Then $\lim_{s\downarrow 0}\Theta^{\cH_{n,1}}_{\Sigma_p}(z,s)=0$ by Theorem \ref{t:main2}. In particular, there exists $s_1>0$ such that \begin{equation}\label{e:m1-1}\Theta^{\cH_{n,1}}_{\Sigma_p}(z,\tfrac{3}{2}s_1)\leq \beta/18.\end{equation}
Since $r_i^{-1}(\overline{A}-x)\rightarrow \Sigma_p$, there exist $y_i\in \overline{A}$ such that $z_i:=(y_i-x)/r_i\rightarrow z$. Replacing each $y_i$ with $y_i'\in A$ such that $|y'_i-y_i|\leq r_i/i$, say, we may assume without loss of generality that $y_i\in A$ for all $i$ (because $\mD{0}{r}\left[r_i^{-1}(\overline{A}-y'_i),r_i^{-1}(\overline{A}-y_i)\right]\leq 1/ri\rightarrow 0$ for all $r>0$). Necessarily, $y_i\rightarrow x$, and thus, there exists $s_2>0$ such that \begin{equation}\label{e:m1-2} \sup_{i\geq 1}\Theta_A^{\cH_{n,d}}(y_i,s)\leq \alpha/2<\alpha\quad\text{for all }s\leq s_2,\end{equation} because $A$ is locally bilaterally well approximated by $\cH_{n,d}$. Now, by quasimonotonicity of bilateral approximation numbers (see Lemma \ref{l:tprop}) and \eqref{e:m1-1}, $$\Theta_{\Sigma_p}^{\cH_{n,1}}(z_i,\tfrac{1}{2}s_1) \leq 2t+2(1+t)\Theta_{\Sigma_p}^{\cH_{n,1}}(z,(1+t)s_1) \leq 2t+3\Theta_{\Sigma_p}^{\cH_{n,1}}(z,\tfrac{3}{2}s_1)\leq 2t+\beta/6$$ whenever $|z_i-z|\leq ts_1\leq \tfrac{1}{2}s_1$. With $t=|z_i-z|/s_1$, this yields  $$\Theta_{\Sigma_p}^{\cH_{n,1}}(z_i,\tfrac{1}{2}s_1)\leq 2|z_i-z|/s_1+\beta/6$$ for all $i$ sufficient large such that $|z_i-z|\leq \tfrac{1}{2}s_1$. Hence, for all  $i$ sufficient large such that $|z_i-z|<s_1/6$ (guaranteeing $z\in \Sigma_p\cap B(z_i,\frac{1}{6}s_1)\neq\emptyset$), \begin{equation*}\begin{split} \Theta_{r_{i}^{-1}(\overline{A}-x)}^{\cH_{n,1}}(z_i,\tfrac{1}{6}s_1) &\leq 3 \mD{z_i}{\frac{1}{2}s_1}\left[\frac{\overline{A}-x}{r_i},\Sigma_p\right]+ 3\Theta_{\Sigma_p}^{\cH_{n,1}}(z_i,\tfrac{1}{2}s_1) \\ &\leq 6 \mD{z}{s_1}\left[\frac{\overline{A}-x}{r_i},\Sigma_p\right]+6|z-z_i|/s_1+\beta/2,\end{split}\end{equation*} where we used the weak quasitriangle inequality in the first line and we used the quasimonotoncity of the relative Walkup-Wets distance in the second line (see Lemma \ref{l:Dprop}). Since $z_i\rightarrow z$ and $r_i^{-1}(\overline{A}-x)\rightarrow \Sigma_p$, we conclude that \begin{equation}\label{e:m1-3}\limsup_{i\rightarrow\infty}\Theta_{A}^{\cH_{n,1}}(y_i,\tfrac{1}{6}r_is_1)=\limsup_{i\rightarrow\infty} \Theta_{r_{i}^{-1}(\overline{A}-x)}^{\cH_{n,1}}(z_i,\tfrac{1}{6}s_1) \leq \frac{2}{3}\beta<\beta.\end{equation} Note that $\frac{1}{6}r_is_1\leq s_2$ for all $i\gg 1$, since $r_i\rightarrow 0$. Therefore, by \eqref{e:m1-0}, \eqref{e:m1-2}, and \eqref{e:m1-3}, we have $y_i\in A_1$ for all sufficiently large $i$. Recalling that   $y_i\rightarrow x$, it follows that $x\in \overline{A_1}$. Since $x\in A\setminus A_1$ was fixed arbitrarily, this proves (vi).

We now aim to prove dimension bounds on $A$ and $A\setminus A_1$ assuming that $A$ is closed and nonempty. Since $\cH_{n,d}$ is a closed, translation invariant approximation class and $\cH_{n,1}$ points are detectable in $\cH_{n,d}$, the set $$\sing{\cH_{n,1}}{\cH_{n,d}}=\{(\Sigma_p)_{\cH_{n,1}^\perp}:\Sigma_p\in \cH_{n,d}\text{ and }0\in(\Sigma_p)_{\cH_{n,1}^\perp}\}$$ is also a local approximation class and $A\setminus A_1$ is locally \emph{unilaterally} well approximated by $\sing{\cH_{n,1}}{\cH_{n,d}}$ by Theorem \ref{singapprox}. By Theorem \ref{t:main2}, applied with $k=1$, the class $\sing{\cH_{n,1}}{\cH_{n,d}}$ is precisely the class $\cS\cH_{n,d}=\{S_p=\Sigma_p\cap|Dp|^{-1}(0):\Sigma_p\in\cH_{n,d}, 0\in S_p\}$ of all singular sets of nonconstant harmonic polynomials of degree at most $d$ that include the origin. Recall from the introduction that \begin{equation*} \Vol\big(\{x\in B(0,1/2):\dist(x,\Sigma_p)\leq r\}\big)\leq (C(n)d)^d\,r\quad\text{for all }\Sigma_p\in\cH_{n,d}\end{equation*} and
\begin{equation*} \Vol\big(\{x\in \ball(0,1/2): \dist(x,S_p)\leq r\}\big) \leq C(n)^{d^2}r^2\quad\text{for all }S_p\in\cS\cH_{n,d}\end{equation*} by work of Naber and Valtorta \cite{nabervaltorta}. Using an elementary Vitali covering argument (e.g., see \cite[(5.4) and (5.6)]{Mattila95}), it follows that $\cH_{n,d}$ has an $(n-1,C(n,d),1)$ covering profile and $\cS\cH_{n,d}$ has an $(n-2,C(n,d),1)$ covering profile in the sense of Definition \ref{coveringprofiles}.

Assume that $A$ is a nonempty closed subset of $\RR^n$. Since $A\setminus A_1$ is relatively closed in $A$ by (v), $A\setminus A_1$ is closed  in $\RR^n$, as well. By Theorem \ref{c:dim3}, $A$ has upper Minkowski dimension at most $n-1$, since $A$ is closed, $A$ is locally unilaterally well approximated by $\cH_{n,d}$, and $\cH_{n,d}$ has an $(n-1,C(n,d),1)$ covering profile. Also, by Theorem \ref{c:dim3}, $A\setminus A_1$ has upper Minkowski dimension at most $n-2$, since $A\setminus A_1$ is closed, $A\setminus A_1$ is locally unilaterally well approximated by $\cS\cH_{n,d}$, and $\cS\cH_{n,d}$ has an $(n-2,C(n,d),1)$ covering profile.
This establishes (viii) and the upper bound in (vii). To wrap up, observe that $A_1$ is nonempty by (vi), $A_1$ is locally closed by (ii), and $A_1$ is locally Reifenberg vanishing by (iii). Therefore, by Reifenberg's topological disk theorem (e.g.~see \cite{davidtoro-holes}), $A_1$ is a topological $(n-1)$-manifold (and more, see Remark \ref{r:reif}). Therefore, $A_1$ has Hausdorff and upper Minkowski dimension at least $n-1$. This completes the proof of (vii).
\end{proof}

By examining the proof that $A_1$ is relatively dense in $A$ in the proof of Theorem \ref{t:main1}, one sees the only essential property about the cones $\cH_{n,1}$ and $\cH_{n,d}$, beyond detectability, was that for every $\Sigma_p\in\cF_{n,k}$ there exist some $z\in\Sigma_p$ such that $\liminf_{s\downarrow 0}\Theta^{\cH_{n,1}}_{\Sigma_p}(z,s)=0$. Thus, abstracting the argument, one obtains the following result.

\begin{theorem} Let $\cT$ and $\cS$ be local approximation classes. Suppose that $\cT$ points are detectable in $\cS$, and  \begin{equation}
 \hbox{for all $S\in\ccS\cap\cT^\perp$ there exists $x\in S$ such that $\liminf_{r\downarrow 0} \Theta_{S}^{\cT}(x,r)=0$.}\end{equation} If $A$ is locally bilaterally well approximated by $\cS$, then the set $A_{\ccT}$ described by Theorem \ref{t:open} is relatively dense in $A$, i.e.~ $\overline{A_{\ccT}}\cap A=A$.
\end{theorem}

\section{Dimension bounds in the presence of good topology}

We now focus our attention on sets $A$ that separate $\RR^n$ into two connected components.  When $A=\Sigma_p$ and $p:\RR^n\rightarrow\RR$ is harmonic, this occurs precisely when the positive set $\Omega^+_p=\{x\in\RR^n:p(x)>0\}$ of $p$ and the negative set $\Omega^-_p=\{x\in\RR^n:p(x)<0\}$ of $p$ are pathwise connected.
To start, let us prove Lemma \ref{l:example} from the introduction, which implies that $\cF_{n,k}$ contains zero sets $\Sigma_p$ that separate $\RR^n$ into two components for all dimensions $n\geq 4$ and for all degrees $k\geq 2$.

\begin{proof}[Proof of Lemma \ref{l:example}] We sketch the argument when $a=b=1$, with the other cases following from an obvious modification. Let $q:\RR^2\rightarrow\RR$ be a homogeneous harmonic polynomial of degree $k\geq 2$. Note that by elementary complex analysis, $q$ can be written as the real part of a complex polynomial $\tilde q:\CC\rightarrow\CC$, $\tilde q(z)=cz^k$. Thus, $\Sigma_q$ is the union of $k$ equiangular lines through the origin and the chambers of $\RR^2\setminus\Sigma_q$ alternate between the positive and negative sets of $q$.
Let $U=(x_1,y_1)$ be any point such that $q(U)>0$. Then $p(U,U)>0$, as well, where $p(W_1,W_2)\equiv q(W_1)+q(W_2)$. To show that the positive set of $p$ is connected, it suffices to show that any point $(V_1,V_2)\in\RR^2\times\RR^2$ such that $p(V_1,V_2)>0$ can be connected to $(U,U)$ by a piecewise linear path in the positive set. If $p(V_1,V_2)>0$, then $q(V_1)>0$ or $q(V_2)>0$, say without loss of generality that $q(V_1)>0$. Then the desired path from $(V_1,V_2)$ to $(U,U)$ is described in Figure 6.1 nearby.
\begin{figure}
\begin{center}\includegraphics[width=.65\textwidth]{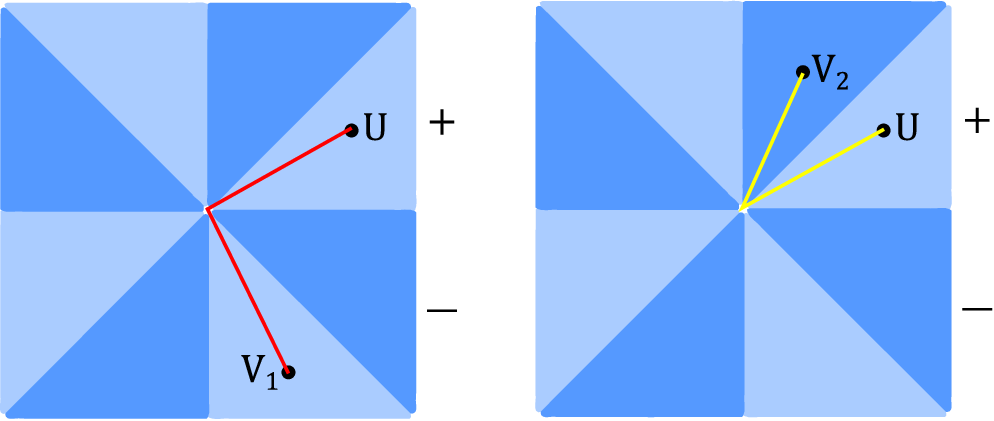}\end{center}%
\caption{Let $q:\RR^2\rightarrow\RR$ denote a nonconstant homogeneous harmonic polynomial (illustrated with degree 4). The light blue cells denote the positive set of $q$ and the medium blue cells denote the negative set of $q$. Suppose that $q(U)>0$, $q(V_1)>0$, and $p(V_1,V_2)>0$, where $p(W_1,W_2)\equiv q(W_1)+q(W_2)$. To move from $(V_1,V_2)$ to $(U,U)$ inside the positive set of $p$, first send $V_2$ to $U$ along the yellow path and then move $V_1$ to $U$ along the red path.}
\end{figure}
A similar argument verifies that the negative set of $p$ is connected and we are done.
\end{proof}

Our goal for the remainder of this section is to prove Theorem \ref{t:main3}, which requires the following notion of non-tangential accessibility due to Jerison and Kenig \cite{jerisonandkenig}.

\begin{definition}[\cite{jerisonandkenig}] \label{ntadomains}
A domain (i.e.~a connected open set) $\Omega \subset \mathbb R^n$ is called \emph{NTA} or \emph{non-tangentially accessible} if there exist constants $M>1$ and $R > 0$
 for which the following are true:
\begin{enumerate}
\item $\Omega$ satisfies the \emph{corkscrew condition}: for all $Q \in \partial \Omega$ and $0 < r < R$, there exists $x\in \Omega\cap B(Q,r)$ such that $\dist(x,\partial\Omega)>M^{-1}r$.
\item $\RR^n\setminus\Omega$ satisfies the corkscrew condition.
\item $\Omega$ satisfies the \emph{Harnack chain condition}: If $x_1, x_2\in \Omega\cap B(Q,r/4)$ for some $Q\in \partial \Omega$ and $0<r<R$, and $\dist(x_1,\partial\Omega)>\delta$, $\dist(x_2,\partial\Omega)>\delta$, and $|x_1-x_2|<2^l\delta$ for some $\delta>0$ and $l\geq 1$, then there exists a chain of no more than $Ml$ overlapping balls  connecting $x_1$ to $x_2$ in $\Omega$ such that for each ball $B=B(x,s)$ in the chain: \begin{align*}&M^{-1}s<\gap(B,\partial\Omega)<Ms, &&\gap(B,\partial\Omega)=\inf_{x\in B}\inf_{y\in\partial\Omega}|x-y|, \\ \diam B> &M^{-1}\min\{\dist(x_1,\partial\Omega),\dist(x_2,\partial\Omega)\}, &&\diam B=\sup_{x,y\in B}|x-y|.\end{align*}
\end{enumerate} We refer to $M$ and $R$ as \emph{NTA constants} of the domain $\Omega$. When $\partial\Omega$ is unbounded, $R=\infty$ is allowed. To distinguish between conditions (i) and (ii), the former may be called the \emph{interior corkscrew condition} and the latter may be called the \emph{exterior corkscrew condition}.
\end{definition}

\begin{remark}\label{r:ntadomains} In the definition of NTA domains, the additional restriction $R=\infty$ when $\Omega$ is unbounded is sometimes imposed (e.g.~ see \cite{kenigtoroannals}, \cite{kenigtorotwophase}, or \cite{kenigpreisstoro}) in order to obtain globally uniform harmonic measure estimates on unbounded domains, but that restriction is not essential in the geometric context of Theorem \ref{t:main3}, and thus, we omit it. \end{remark}

 An essential feature of NTA domains that we need below is that the NTA properties persist under limits (with slightly different constants). When $\Gamma_i=r_i^{-1}(\partial\Omega -Q_i)$ is a sequence of pseudoblowups of the boundary $\partial\Omega$ of a 2-sided NTA domain $\Omega\subset\RR^n$ for some $Q_i\in\partial\Omega$ and $r_i>0$ such that $Q_i\rightarrow Q\in\partial\Omega$ and $r_i\downarrow 0$, the following lemma  is due to Kenig and Toro \cite[Theorem 4.1]{kenigtorotwophase}; also see \cite[Lemma 1.5]{AM-tangents} for a recent variant on uniform domains due to Azzam and Mourgoglou. For the proof of Lemma \ref{closedunderblowup}, see Appendix B below.

\begin{lemma}\label{closedunderblowup}
 Suppose that $\Gamma_i \subset \RR^n$ is a sequence of closed sets such that $\RR^n\setminus \Gamma_i = \Omega^+_i\cup\Omega^-_i$ is the union of complimentary NTA domains $\Omega_i^+$ and $\Omega_i^-$ with NTA constants $M$ and $R$ independent of $i$. If $\Gamma_i \rightarrow \Gamma\neq\emptyset$ in the Attouch-Wets topology, then $\RR^n\setminus \Gamma = \Omega^+\cup\Omega^-$ is the union of complementary NTA domains $\Omega^+$ and $\Omega^-$ with NTA constants $2M$ and $R$.
\end{lemma}

In the remainder of this section, we work with subclasses of $\cH_{n,d}$ and $\cF_{n,k}$ whose zero sets
 $\Sigma_p$ separate $\RR^n$ into two distinct NTA components with uniform NTA constants.

\begin{definition}[2-sided NTA restricted classes $\cH_{n,d}^*$, $\cH_{n,d}^{**}$, $\cF_{n,k}^*$, $\cF_{n,k}^{**}$] \label{handfclass} For all $n\geq 2$ and $d\geq 1$, let $\cH_{n,d}^*$ denote the collection of all $\Sigma_p\in\cH_{n,d}$ such that $\Omega^\pm_p=\{x\in\RR^n:\pm p(x) > 0\}$ are NTA domains with NTA constants $M^*=M$ and $R^*=\infty$ for some fixed $M>1$. (We deliberately suppress the choice of $M^*$ from the notation.) Also, let $\cH_{n,d}^{**}$ denote the collection of all $\Sigma_p\in\cH_{n,d}$ such that $\Omega^\pm_p$ are NTA domains with NTA constants $M^{**}=2M^*$ and $R^{**}=\infty$. Finally, set $\cF_{n,k}^*=\cH_{n,k}^*\cap\cF_{n,k}$ and $\cF_{n,k}^{**}=\cH_{n,k}^{**}\cap \cF_{n,k}$ for all $k\geq 1$.\end{definition}

\begin{remark} \label{closedclasses} The classes $\starH_{n,d}$ (hence $\cH_{n,d}^{**}$) and $\starF_{n,k}$ (hence $\cF_{n,k}^{**}$) are local approximation classes (see Definition \ref{d:big}), because $R^*=\infty$, and it is apparent that $\starH_{n,d}$ is translation invariant in the sense that $\Sigma_p-x\in\starH_{n,d}$ for all $\Sigma_p\in\starH_{n,d}$ and $x\in \Sigma_p$. Hence $\overline{\starH_{n,d}}$ is also translation invariant. By Corollary \ref{c:closed} and Lemma \ref{closedunderblowup}, $\overline{\starH_{n,d}}\subseteq \cH_{n,d}^{**}$ and $\overline{\starF_{n,k}} \subseteq \cF_{n,k}^{**}$. Since $\cH_{n,k}$ points are detectable in $\cH_{n,d}$ for all $1\leq k\leq d$ by Corollary \ref{c:main2} and $\cH_{n,d}^*\subseteq\cH_{n,d}$, we have $\cH_{n,k}$ points are detectable in $\cH_{n,d}^*$, as well. Finally, we reiterate that $\cF_{n,k}^*$ is nonempty for some $M^*>1$ if and only if $k=1$ and $n\geq 2$; $k\geq 2$ is even and $n\geq 4$; or, $k\geq 3$ is odd and $n\geq 3$. See Remark \ref{r:sharp}. The assertion that the interior of the two connected components of $\RR^n\setminus\Sigma_p$ are NTA domains when $n=3$ and $p=p(x,y,z)$ is Szulkin's polynomial (or any of Lewy's odd degree polynomials) and when $n=4$ and $p=p(x_1,y_1,x_2,y_2)$ is the zero set of one of the polynomials from Lemma \ref{l:example} follows from the fact that in each case $\Sigma_p\cap \partial B(0,1)$ is a smooth hypersurface in the unit sphere and $\Sigma_p$ is a cone.
\end{remark}

The following technical proposition, alluded to in the introduction after the statement of Theorem \ref{t:main3}, is a consequence of Lemma \ref{closedunderblowup}.

\begin{lemma}\label{puttingthestarin} Suppose that $A\subseteq\RR^n$ is closed and $\RR^n\setminus A = \Omega^+\cup\Omega^-$ is a union of complementary NTA domains. If $A$ is locally bilaterally well approximated by $\cH_{n,d}$ for some $n\geq 2$ and $d\geq 1$, then $A$ is locally bilaterally well approximated by $\starH_{n,d}$ for some $M^*>1$ depending only on the NTA constants of $\Omega^+$ and $\Omega^-$.
\end{lemma}

\begin{proof} Suppose that $A$ is closed, $A$ is locally bilaterally well approximated by $\cH_{n,d}$, and $\RR^n\setminus A=\Omega^+\cup\Omega^-$ is a union of complementary NTA domains with uniform NTA constants $M$ and $R$. On one hand, $\PsTan(A,x)\subseteq\overline{\cH_{n,d}}=\cH_{n,d}$ for all $x\in A$ by Theorem \ref{t:tangent-well} and Corollary \ref{c:closed}. On the other hand, for every $x\in A$ and $r>0$, the set $(A-x)/r=\Omega^+_{x,r}\cup\Omega^-_{x,r}$ is a union of complementary NTA domains $\Omega^+_{x,r}$ and $\Omega^-_{x,r}$ with NTA constants $M_{x,r}=M$ and $R_{x,r}=R/r$. Thus, every pseudotangent set $T=\lim_{i\rightarrow 0}(A-x_i)/r_i\in\PsTan(A,x)$ separates $\RR^n$ into two NTA domains with NTA constants $M_T=2M$ and $R_T=\infty$ by Lemma \ref{closedunderblowup}, since $R_{x_i,r_i}=R/r_i\rightarrow\infty$ as $r_i\rightarrow 0$. Therefore, $\PsTan(A,x)\subseteq\cH_{n,d}^*$ for every $x\in A$ with $M^*=2M$. By Theorem \ref{t:tangent-well}, it follows that $A$ is locally bilaterally well approximated by $\cH_{n,d}^*$, as desired.
\end{proof}

In view of Lemma \ref{puttingthestarin}, Theorem \ref{t:main3} is a special case of the following theorem.

\begin{theorem}\label{t:main4} Let $n\geq 2$, $d\geq 2$, and $M^*>1$. If $A\subseteq\RR^n$ is closed and locally bilaterally well approximated by $\cH_{n,d}^*$, then  \begin{enumerate}
\item[(i)]  $A\setminus A_1=A_2\cup\dots \cup A_{d}$ has upper Minkowski dimension at most $n-3$; and,
\item[(ii)] the ``even singular set" $A_2\cup A_4\cup A_6 \cup \cdots$ has Hausdorff dimension at most $n-4$.
\end{enumerate}\end{theorem}

To prove Theorem \ref{t:main4} using the technology of \cite{localsetapproximation}, we need to show the existence of ``covering profiles" (see Definition \ref{coveringprofiles}) for the classes $\sing{\cH_{n,1}}{\overline{\cH_{n,d}^*}}$ and $\sing{\cH_{n,d-1}}{\overline{\cH_{n,d}^*}}$ (see Definition \ref{sing}), which are well defined because $\overline{\cH_{n,d}^*}$ is translation invariant and $\cH_{n,k}$ points are detectable in $\cH_{n,d}^*$ by Remark \ref{closedclasses}. The following lemma proves the existence of good covering profiles for $\sing{\cH_{n,k-1}}{\overline{\cH_{n,k}^*}}$ for all degrees $k\geq 2$.

\begin{lemma}\label{l:subspace} Let $k\geq 2$ and assume that $n+(k\hbox{\rm\ mod }2)\geq 4$. For every $k$ homogeneous harmonic polynomial $p:\RR^n\rightarrow\RR$ such that $\RR^n\setminus\Sigma_p$ has two connected components, $$(\Sigma_p)_{\cH_{n,k-1}^\perp}=\{x\in \Sigma_p:{\liminf_{r\rightarrow 0}}\Theta_{\Sigma_p}^{\cH_{n,k-1}}(x,r)>0\}$$ is a linear subspace $V$ of $\RR^n$ with $\dim V\leq n-4+(k\hbox{\rm\ mod }2)$. In particular, $$\sing{\cH_{n,k-1}}{\overline{\starH_{n,k}}}=\left\{(\Sigma_p)_{\cH_{n,k-1}^\perp}: \Sigma_p\in\overline{\starH_{n,k}},\ 0\in(\Sigma_p)_{\cH_{n,k-1}^\perp}\right\}$$ admits an $(n-4+(k\hbox{\rm\ mod }2),C(n),1)$ covering profile.
\end{lemma}

\begin{proof} Suppose that $k$ and $n$ satisfy the hypothesis of the lemma and let $p:\RR^n\rightarrow\RR$ be a $k$ homogeneous harmonic polynomial. We will show that $(\Sigma_p)_{\cH_{n,k-1}}^\perp$ coincides with $$V=\{x_0\in \RR^n: p(x+x_0)=p(x)\hbox{ for all }x\in\RR^n\},$$ which is a linear subspace of $\RR^n$ because $p$ is $k$ homogeneous. To start, note that \begin{align*}
x_0\in (\Sigma_p)_{\cH_{n,k-1}^\perp}&\Longleftrightarrow \partial^\alpha p(x_0)=0\quad\text{for all }|\alpha|\leq k-1\\ &\Longleftrightarrow p(x+x_0)\equiv q(x)\text{ for some $q$, where $q:\RR^n\rightarrow\RR$ is $k$ homogeneous},\end{align*} where the first equivalence holds by Theorem \ref{t:main2} and the second equivalence holds by Taylor's theorem. Hence $V\subseteq (\Sigma_p)_{\cH_{n,k-1}^\perp}$, since $p$ is $k$ homogeneous. Conversely, using the homogeneity of $p$ and $q$, at any $x_0\in (\Sigma_p)_{\cH_{n,k-1}^\perp}$ we obtain $$p(x+x_0)=q(x)=\lambda^k q(x/\lambda)=\lambda^k p(x/\lambda +x_0)=p(x+\lambda x_0)\quad\text{for all } \lambda\in\RR\setminus\{0\}.$$ Letting $\lambda\rightarrow 0$, we conclude that $p(x+x_0)=p(x)$ for all $x\in\RR^n$ whenever $x\in(\Sigma_p)_{\cH_{n,k-1}^\perp}$. Thus, $(\Sigma_p)_{\cH_{n,k-1}^\perp}\subseteq V$, as well.

To continue, suppose that $\Sigma_p$ separates $\RR^n$ into two components. Let $\tilde p:V^\perp\rightarrow\RR$ be the image of $p$ under the quotient map $\RR^n\rightarrow\RR^n/V\cong V^\perp$. Because $V$ is the space of invariant directions for $p$, the map $\tilde p$ is still a degree $k$ homogenous harmonic polynomial (in orthonormal coordinates for $V^\perp$) and $$\Sigma_p=\Sigma_{\tilde p}\oplus V=\{x+v:x\in\Sigma_{\tilde p}\subseteq V^\perp, v\in V\}.$$ Hence $\Sigma_{\tilde p}$ separates $V^\perp$ into two components, since $\Sigma_p$ separates $\RR^n$ into two components. It follows that $\dim V^\perp\geq 4$, if $k\geq 2$ is even, and $\dim V^\perp\geq 3$, if $k\geq 3$ is odd; e.g., see the paragraph immediately preceding the statement of Lemma \ref{l:example}. Therefore, $\dim V\leq n-4$, if $k\geq 2$ is even, and $\dim V\leq n-3$, if $k\geq 3$ is odd.

Finally, by Theorem \ref{t:main2}, Remark \ref{closedclasses}, and the first part of the lemma,  $$\sing{\cH_{n,k-1}}{\overline{\starH_{n,k}}}=\left\{(\Sigma_p)_{\cH_{n,k-1}^\perp}: \Sigma_p\in\overline{\starF_{n,k}}\right\}\subseteq\left\{(\Sigma_p)_{\cH_{n,k-1}^\perp}: \Sigma_p\in\cF_{n,k}^{**}\right\}\subseteq\bigcup_{i=0}^{j} G(n,i),$$ where $j=n-4$, if $k\geq 2$ is even, and $j=n-3$, if $k\geq 3$ is odd. Here each $G(n,i)$ denotes the Grassmannian of dimension $i$ linear subspaces of $\RR^n$, which possesses an $(i,C(i),1)$ covering profile; that is, $V\cap B(0,r)$ can be covered by $C(i)s^{-i}$ balls $B(v_i,sr)$ centered in $V\cap B(0,r)$ for all planes $V\in G(n,i)$, $r>0$, and $0<s\leq 1$. (For example, this follows from the fact that Lebesgue measure of any ball of radius $r$ in $\RR^i$ is proportional to $r^i$.) It follows that the class $\sing{\cH_{n,k-1}}{\overline{\starH_{n,k}}}$ has an $(n-4,C(n),1)$ covering profile when $k\geq 2$ is even, and $\sing{\cH_{n,k-1}}{\overline{\starH_{n,k}}}$ has an $(n-3,C(n),1)$ covering profile when $k\geq 3$ is odd.
\end{proof}

The covering profiles for $\sing{\cH_{n,k-1}}{\overline{\starH_{n,k}}}$ from  Lemma \ref{l:subspace} will enable us to prove (ii) in Theorem \ref{t:main4} and also to prove that $A\setminus A_1$ has Hausdorff dimension at most $n-3$. However, to show that $A\setminus A_1$ has upper Minkowski dimension at most $n-3$, we need to find covering profiles for $\sing{\cH_{n,1}}{\overline{\cH_{n,d}^*}}$, whose existence does not automatically follow from the covering profiles in Lemma \ref{l:subspace}. To proceed, we use the quantitative stratification and volume estimates for singular sets of harmonic functions developed by Cheeger, Naber, and Valtorta in \cite{cheegernabervaltorta}. The following description of the stratification combines several definitions from \S1 of \cite{cheegernabervaltorta}; see \cite[Definition 1.4, Definition 1.7, Remark 1.8, and Definition 1.9]{cheegernabervaltorta}.

\begin{definition}[\cite{cheegernabervaltorta}; quantitative stratification by symmetry] A smooth function $u:\RR^n\rightarrow\RR$ is called \emph{$0$-symmetric} if $u$ is a homogeneous polynomial and $u$ is called \emph{$k$-symmetric} if $u$ is $0$-symmetric and there exists a $k$-dimensional subspace $V$ such that $$u(x+y)=u(x)\quad\text{for all $x\in\RR^n$ and $y\in V$}.$$ For all smooth $u:B(0,1)\rightarrow\RR$, and for all $x\in B(0,1-r)$, define $$T_{x,r}u(y) =\frac{u(x+ry)- u(x)}{\left(\fint_{\partial B(0,1)}|u(x+rz) - u(x)|^2\,d\sigma(z)\right)^{1/2}}\quad\text{for all }y\in B(0,1).$$ (If the denominator vanishes, set $T_{x,r}=\infty$.)
A harmonic function $u: B(0,1)\rightarrow \mathbb R$ is called \emph{$(k, \varepsilon, r, x)$-symmetric} if there exists a harmonic $k$-symmetric function $p$ with $\int_{\partial B(0,1)} |p|^2\, d\sigma = 1$ such that $$\fint_{B(0,1)}|T_{x,r}u - p|^2 < \varepsilon.$$ For all harmonic $u:B(0,1)\rightarrow\RR$, define the \emph{$(k,\eta,r)$-effective singular stratum} by $$\mathscr{S}^k_{\eta, r}(u) = \{x\in B(0,1): \hbox{$u$ is not $(k+1, \eta, s, x)$-symmetric for all $s\geq r$}\}.$$

\end{definition}

For harmonic functions, \cite[Theorem 1.10]{cheegernabervaltorta} gives the following Minkowski type estimates for effective singular strata. In the statement, $N(1,0,u)$ denotes Almgren's frequency function with $r=1$, $x_0=0$, and $f=u$ (recall Definition \ref{almgrenfrequency} above).

\begin{theorem}[\cite{cheegernabervaltorta}]\label{cnvestimate}
If $u: B(0,1)\rightarrow \mathbb R$ is a harmonic function with $u(0)=0$ and $N(1,0,u) \leq \Lambda<\infty$, then for every $\eta>0$ and $k\leq n-2$,
\begin{equation}\label{volumeestimateonsymmetricpoints}
\mathrm{Vol}(\{x\in B(0,1/2):\dist(x,\mathscr{S}^k_{\eta,r}(u))<r\}) \leq C(n,\Lambda, \eta)r^{n-k-\eta}.
\end{equation}
\end{theorem}

We now show that if $\eta$ is small enough depending on $n$, $d$, and $M^*$, then the singular set of $\Sigma_p\in\starH_{n,d}$ is contained in $\mathscr S^{n-3}_{\eta, r}(p)$.

\begin{lemma}\label{criticalinsidestrata} For all $n\geq 2$, $d\geq 2$, and $M^*>1$, there exists $\overline{\eta}>0$ with the following property. If $\Sigma_p \in \starH_{n,d}$, $x_0\in\Sigma_p$, and $p$ is $(n-2, \eta, r, x_0)$-symmetric for some $\eta\in(0,\overline{\eta})$ and $r>0$, then $x_0$ is an $\cF_{n,1}$ point of $\Sigma_p$. Consequently, the set of all singular points of $\Sigma_p$ (that is, $\cF_{n,2}\cup\dots\cup\cF_{n,d}$ points of $\Sigma_p$) belongs to $\mathscr {S}^{n-3}_{\eta,r}(p)$ for all $\eta\in(0,\overline{\eta})$ and $r>0$.
\end{lemma}

\begin{proof}
Let $n\geq 2$, $d\geq 2$, and $M^*>1$ be given. Assume in order to obtain a contradiction that for all $i\geq 1$, there exist $\Sigma_{p_i} \in \starH_{n,d}$, $\eta_i<1/i$, $x_i\in\Sigma_{p_i}$, and $r_i>0$ such that $p_i$ is $(n-2,\eta_i,r_i,x_i)$-symmetric and $x_i$ is not an $\cF_{n,1}$ point of $\Sigma_{p_i}$. Equivalently, by Theorem \ref{t:main2}, $D p_i(x_i)=0$. That is, the Taylor expansion for $p_i$ at $x_i$ has no nonzero linear terms. By definition of almost symmetry, there exist $(n-2)$-symmetric homogenous harmonic polynomials $h_i$ such that $\fint_{\partial B(0,1)} |h_i|^2\, d\sigma = 1$ and \begin{equation}\label{e:L2diff} \fint_{B(0,1)} \left|T_{x_i, r_i} p_i - h_i\right|^2 < \frac{1}{i}.\end{equation}  As everything is translation, dilation, and rotation invariant, we may assume without loss of generality that for all $i\geq 1$, $x_i=0$, $r_i=1$, and $h_i(y_1,y_2,\dots,y_n)=h_i(y_1,y_2,0,\dots,0)$ for all $y\in \mathbb R^n$. To ease notation, let us abbreviate $q_i\equiv T_{0,1}p_i$. We note that \begin{equation}\label{e:qnorm}\|q_i\|_{L^2(B(0,1))}\sim_{n,d} \|q_i\|_{L^2(\partial B(0,1))}\sim_{n,d}1\quad\text{for all }i\geq 1,\end{equation} where the first comparison holds by Lemma \ref{l:L2L2} and the second comparison holds by the definition of $T_{0,1}p_i$.

We now claim that $\deg h_i \leq d$ for all $i$ sufficiently large. To see this, suppose to the contrary that $l:=\deg h_i>d$ for some $i\geq 1$. Recalling both that spherical harmonics of different degrees are orthogonal on spheres centered at the origin and that  $h_i$ is $l$ homogeneous with $l>\deg q_i$, we have $$1\sim_{n,d}\|q_i\|_{L^2(B(0,1))}^2\lesssim_{n,d}\fint_{B(0,1)} \left(q_i^2+h_i^2\right)=\fint_{B(0,1)} |q_i-h_i|^2<\frac{1}{i}$$ by \eqref{e:L2diff} and \eqref{e:qnorm}. This is impossible if $i$ is sufficient large depending only on $n$ and $d$. Thus, $\deg h_i\leq d$ for all $i$ sufficient large, as claimed. In particular, \begin{equation}\label{e:hnorm}\|h_i\|_{L^2(B(0,1))}\sim_{n,d} \|h_i\|_{L^2(\partial B(0,1))}\sim_{n,d} 1\quad\text{for all }i\gtrsim_{n,d} 1.\end{equation}

By \eqref{e:qnorm}, \eqref{e:hnorm}, Lemma \ref{l:height}, and Lemma \ref{supcomparabletoaverage}, we conclude that the heights $H(q_i)\sim_{n,d} 1$ and $H(q_i)\sim_{n,d} 1$ for all sufficiently large $i$. Therefore, by passing to a subsequence of the pair $(q_i,h_i)_{i=1}^\infty$ (which we relabel), we may assume that $q_i\rightarrow q$ in coefficients and $h_i\rightarrow h$ in coefficients for some nonconstant harmonic polynomials $q$ and $h$ of degree at most $d$. On one hand, we have $\Sigma_q\in\overline{\cH_{n,d}^*}\subseteq \cH_{n,d}^{**}$  by Lemma \ref{convergenceofzerosets} and $Dq(0)=0$, since $Dq_i(0)=0$ for all $i$. Hence $q$ has degree at least $2$. On the other hand, we have $h$ is homogeneous and $h(y_1,y_2,\dots,y_n)=h(y_1,y_2,0,\dots,0)$ for all $y\in\RR^n$, because the same are true of the polynomial $h_i$ for all $i\gtrsim_{n,d} 1$.

We are now ready to obtain a contradiction. Since $q_i\rightarrow q$ and $h_i\rightarrow h$ uniformly on compact sets (see Remark \ref{r:coeff}), we have $q\equiv h$ by \eqref{e:L2diff}. Thus, $\Sigma_q\in \cF_{n,k}^{**}$ for some $2\leq k\leq d$---in particular, $\Sigma_q$ is the zero set of a homogeneous harmonic polynomial of degree at least 2 that separates $\RR^n$ into 2 components---and $q$ depends on at most 2 variables. No such polynomial $q$ exists (e.g.~ see Remark \ref{closedclasses})! Therefore, for all $n\geq 2$, $d\geq 2$, and $M^*>1$, there exists $\overline{\eta}>0$ such that if $\Sigma_p \in \starH_{n,d}$, $x_0\in\Sigma_p$, and $p$ is $(n-2, \eta, r, x_0)$-symmetric for some $\eta\in(0,\overline{\eta})$ and $r>0$, then $x_0$ is an $\cF_{n,1}$ point of $\Sigma_p$. Consequently, if $\Sigma_p\in\starH_{n,d}$ and $x_0\in\Sigma_p$ belongs to the singular set of $p$, then $p$ is not $(n-2,\eta,r,x_0)$ symmetric for all $\eta\in (0,\overline{\eta})$ and $r>0$. By definition of the singular strata, we conclude that for all $\Sigma_p\in\cH_{n,d}^*$ the set of all singular points of $\Sigma_p$ belongs to $\mathscr {S}^{n-3}_{\eta,r}(p)$ for all $\eta\in(0,\overline{\eta})$ and $r>0$. \end{proof}

At last, we are ready to prove Theorem \ref{t:main4} and Theorem \ref{t:main3}.

\begin{proof}[Proof of Theorem \ref{t:main4} and Theorem \ref{t:main3}] As noted earlier, Theorem \ref{t:main4} implies Theorem \ref{t:main3} by Lemma \ref{puttingthestarin}. Thus, it suffices to establish the former. Assume $A\subseteq\RR^n$ is closed and locally bilaterally well approximated by $\cH_{n,d}^*$ for some $M^*>1$. Then $A$ can be written as $A=A_1\cup A_2\cup\dots\cup A_d$ according to Theorem \ref{t:main1}. In particular, $U_k=A_1\cup\dots \cup A_k$ is relatively open in $A$ and locally bilaterally well approximated by $\cH_{n,k}$ for all $1\leq k\leq d$. Hence $U_k$ is also locally bilaterally well approximated by $\cH_{n,k}^{**}$ for all $1\leq k\leq d$, because $\PsTan(A,x)\subseteq\overline{\cH_{n,d}^*}\cap\cH_{n,k}\subseteq\cH_{n,k}^{**}$ for all $x\in U_k$ by Theorem \ref{t:tangent-well} and Remark \ref{closedclasses}. Also, $A\setminus A_1$ is closed in $\RR^n$, because $A_1$ is relatively open in $A$ and $A$ is closed in $\RR^n$, and $A_k$ is $\sigma$-compact for each $k\geq 1$, because $A_k$ is relatively closed in $U_k$, $U_k$ is relatively open in $A$, and $A$ is closed in $\RR^n$. Our goal is to prove that (i) $\udim_M A\setminus A_1 \leq n-3$ and (ii) $\dim_H A_{k}\leq n-4$ for all even $k\geq 2$.

We begin with a proof of (i). By Remark \ref{closedclasses}, $\overline{\cH_{n,d}^{**}}$ is translation invariant and $\cH_{n,1}$ points are detectable in $\cH_{n,d}^{**}$. Thus, $A\setminus A_1$ is locally unilaterally well approximated by $\sing{\cH_{n,1}}{\overline{\cH_{n,d}^{**}}}$ by Theorem \ref{singapprox}. By Lemma \ref{criticalinsidestrata}, Lemma \ref{frequencyofpolynomials}, and Theorem \ref{cnvestimate}, the class $\sing{\cH_{n,1}}{\overline{\cH_{n,d}^{**}}}$ admits an $(n-3+\eta,C(n,d,\eta,M^{**}),1)$ covering profile for all $\eta>0$. Thus, since $A\setminus A_1$ is closed, we have $\udim_M A\setminus A_1\leq n-3+\eta$ for all $\eta>0$ by Theorem \ref{c:dim3}. Letting $\eta\downarrow 0$, we conclude $\udim_M A\setminus A_1\leq n-3$, as desired.

We now prove (ii). Let $k\geq 2$ be even. By Remark \ref{closedclasses}, $\overline{\cH_{n,k-1}^{**}}$ is translation invariant and $\cH_{n,k-1}$ points are detectable in $\cH_{n,k}^{**}$. Thus, $A_k=U_k\setminus U_{k-1}$ is locally unilaterally well approximated by $\sing{\cH_{n,k-1}}{\overline{\cH_{n,k}^{**}}}$ by Theorem \ref{singapprox}. By Lemma \ref{l:subspace}, the class  $\sing{\cH_{n,k-1}}{\overline{\cH_{n,k}^{**}}}$ admits an $(n-4,C(n),1)$ covering profile. Thus,  since $A_k$ is $\sigma$-compact, we have $\dim_H A_k\leq n-4$ by Theorem \ref{c:dim5}, as desired. Because Hausdorff dimension is stable under countable unions, $\dim_H A_2\cup A_4\cup\cdots\leq n-4$, as well.
\end{proof}

\section{Boundary structure in terms of interior and exterior harmonic measures}

Harmonic measure arises in classical analysis from the solution of the Dirichlet problem and in probability as the exit distribution of Brownian motion. For nice introductions to harmonic measure, see the books of Garnett and Marshall \cite{garnettandmarshall} and M\"orters and Peres \cite{brownian}. One of our motivations for this work is the desire to understand the extent to which the structure of the boundary of a domain in $\RR^n$, $n\geq 2$, is determined by the relationship between harmonic measures in the interior and the exterior of the domain. This problem can be understood as a free boundary regularity problem for harmonic measure. For an in-depth introduction to free boundary problems for harmonic measure, see the book of Capogna, Kenig, and Lanzani \cite{capognakeniglanzani}.

Given a simply connected domain $\Omega\subset\RR^2$,
bounded by a Jordan curve, let $\omega^+$ and $\omega^-$ denote the harmonic measures associated to $\Omega^+=\Omega$ and $\Omega^-=\RR^2\setminus\overline{\Omega}$, respectively, which are supported on their common boundary $\partial\Omega=\partial\Omega^+=\partial\Omega^-$. Together, the theorems of McMillan, Makarov, and Pommerenke (see \cite[Chapter VI]{garnettandmarshall}) show that $$\omega^+\ll\omega^-\ll\omega^+\quad\Longrightarrow\quad \omega^+\ll\cH^1|_G\ll\omega^+\text{ and }\omega^-\ll \cH^1|_G\ll\omega^-$$ for some set $G\subseteq\partial\Omega$ with $\sigma$-finite 1-dimensional Hausdorff measure and $\omega^\pm(\partial\Omega\setminus G)=0$; furthermore, in this case, $\partial\Omega$ possesses a unique tangent line at $Q$ for $\omega^\pm$-a.e.~ $Q\in\partial\Omega$. Here $\cH^s$ denotes the $s$-dimensional Hausdorff measure of sets in $\RR^n$.
 Motivated by this result, Bishop \cite{Bishop-questions} asked
whether if on a domain in
$\RR^n$, $n\ge 3$, \begin{equation}\label{e:bishop} \omega^+\ll\omega^-\ll\omega^+\quad\Longrightarrow\quad \omega^+\ll\cH^{n-1}|_G\ll\omega^+\text{ and }\omega^-\ll \cH^{n-1}|_G\ll\omega^-\end{equation} for some  $G\subseteq\partial\Omega$ with $\sigma$-finite $(n-1)$-dimensional Hausdorff measure and $\omega^\pm(\partial\Omega\setminus G)=0$.
In \cite{kenigpreisstoro}, Kenig, Preiss, and Toro proved that when $\Omega^+=\Omega\subset\RR^n$ and $\Omega^-=\RR^n\setminus\overline{\Omega}$ are NTA domains in $\RR^n$, $n\geq 3$, the mutual absolute continuity of $\omega^+$ and $\omega^-$ on a set $E\subseteq\partial\Omega$ implies that
$\omega^\pm|_E$ has \emph{upper Hausdorff dimension} $n-1$: there exists a set $E'\subseteq E$ of Hausdorff dimension $n-1$ such that $\omega^\pm(E\setminus E')=0$, and $\omega^\pm(E\setminus E'')>0$ for every set $E''\subset E$ with $\dim_H E''<n-1$. Moreover, in this case $\omega^\pm|_E\ll \cH^{n-1}|_E\ll \omega^\pm|_E$ provided that $\cH^{n-1}|_{\partial\Omega}$ is locally finite (see Badger \cite{badger-null}, also \cite[Remark 6.19]{badgerflatpoints}).
However, at present it is still unknown whether or not \eqref{e:bishop} holds on domains for which $\cH^{n-1}|_{\partial\Omega}$ is not locally finite. For some related inquiries, see the work of Lewis, Verchota, and Vogel \cite{LVV}, Azzam and Mourgoglou \cite{AM-tangents}, Bortz and Hofmann \cite{bortzhofmann}, and the references therein.

\begin{remark}[Added in February 2017] Several months after the first version of this paper appeared on the arXiv in September 2015, a solution to Bishop's conjecture \eqref{e:bishop} was furnished by Azzam, Mourgoglou, and Tolsa \cite{amt-twophase} and by Azzam, Mourgoglou, Tolsa, and Volberg \cite{amtv-twophase}. An important tool in these works is a new ``bounded Riesz transform" to ``uniform rectifiability" criterion of Girela-Sarri\'on and Tolsa \cite{gs-t}.
\end{remark}

Finer information about the structure and size of the boundary under more stringent assumptions on the relationship between $\omega^+$ and $\omega^-$ has been obtained in \cite{kenigtorotwophase}, \cite{badgerharmonicmeasure}, \cite{badgerflatpoints}, \cite{localsetapproximation}, and \cite{engelsteintwophase}. We summarize these results in Theorem \ref{previousvmoresults} after recalling the definition of the space $\VMO(d\omega)$ of functions of vanishing mean oscillation, which extends the space of uniformly continuous bounded functions on $\partial\Omega$.

\begin{definition}[{\cite[Definition 4.2 and Definition 4.3]{kenigtorotwophase}}] Let $\Omega\subset\RR^n$ be an NTA domain (with the NTA constant $R=\infty$ when $\partial\Omega$ is unbounded) equipped with harmonic measure $\omega$. We say that  $f \in L^2_{\mathrm{loc}}(d\omega)$ belongs to $\BMO(d\omega)$ if and only if \begin{equation*}\label{bmodef}
\sup_{r > 0} \sup_{Q\in \partial \Omega} \left(\fint_{B(Q,r)} |f-f_{Q,r}|^2\, d\omega\right)^{1/2} < \infty,
\end{equation*}
where $f_{Q,r} = \fint_{B(Q,r)}f \, d\omega$ denotes the average of $f$ over the ball. We denote by $\VMO(d\omega)$ the closure in $\BMO(d\omega)$ of the set of uniformly continuous bounded functions on $\partial\Omega$.
\end{definition}

\begin{theorem}[{\cite{kenigtorotwophase,badgerharmonicmeasure,badgerflatpoints,localsetapproximation,engelsteintwophase}}] Assume \label{previousvmoresults} that $\Omega^+=\Omega\subset\RR^n$ and $\Omega^-=\RR^n\setminus\overline{\Omega}$ are NTA domains (with the NTA constant $R=\infty$ when $\partial\Omega$ is unbounded), equipped with harmonic measures $\omega^\pm$ on $\Omega^\pm$. If $\omega^+\ll\omega^-\ll\omega^+$ and the Radon-Nikodym derivative $f=d\omega^-/d\omega^+$ satisfies $\log f\in \VMO(d\omega^+)$, then the boundary $\partial\Omega$ satisfies the following properties. \begin{enumerate}
\item[{\cite{kenigtorotwophase}}] There exist $d\geq 1$ and $M^*>1$ depending on at most $n$ and the NTA constants of $\Omega^+$ and $\Omega^-$ such that $\partial\Omega$ is locally bilaterally well approximated by $\cH^*_{n,d}$.

\item[{\cite{badgerharmonicmeasure}}]$\partial\Omega$ can be partitioned into disjoint sets $\Gamma_k$, $1\leq k\leq d$, such that $x\in \Gamma_k$ if and only if $x$ is an $\cF_{n,k}$ point of $\partial\Omega$. Moreover, $\Gamma_1$ is dense in $\partial\Omega$ and $\omega^\pm(\partial\Omega\setminus \Gamma_1)=0$.
\item[{\cite{badgerflatpoints}}] $\Gamma_1$ is relatively open in $\partial\Omega$, $\Gamma_1$ is locally bilaterally well approximated by $\cH_{n,1}$, and $\Gamma_1$ has Hausdorff dimension $n-1$.
\item[{\cite{localsetapproximation}}] $\partial\Omega$ has upper Minkowski dimension $n-1$ and $\partial\Omega\setminus \Gamma_1=\Gamma_2\cup\dots\cup\Gamma_d$ has upper Minkowski dimension at most $n-2$.
\item[{\cite{engelsteintwophase}}] If $\log f\in C^{l,\alpha}$ for some $l\geq 0$ and $\alpha>0$ (resp.~$\log f\in C^\infty$, $\log f$ real analytic), then $\Gamma_1$ is a $C^{l+1,\alpha}$ (resp.~ $C^\infty$, real analytic) $(n-1)$-dimensional manifold.\end{enumerate}
\end{theorem}

\begin{remark} The statements from \cite{kenigtorotwophase} and \cite{badgerharmonicmeasure} recorded in Theorem \ref{previousvmoresults} were obtained by showing that the pseudotangent measures of the harmonic measures $\omega^\pm$ of $\Omega^\pm$ are ``polynomial harmonic measures"  in \cite{kenigtorotwophase} and by studying the ``separation at infinity" of cones of polynomial harmonic measures associated to polynomials of different degrees in \cite{badgerharmonicmeasure} (also see \cite{kenigpreisstoro}). The statements from \cite{badgerflatpoints} and \cite{localsetapproximation} are forerunners to and motivated the statement and proof of Theorem \ref{t:main1} in this paper. However, we wish to emphasize that the structure theorem \cite[Theorem 5.10]{badgerflatpoints} and dimension estimate on the singular set $\partial\Omega\setminus\Gamma_1$ in \cite[Theorem 9.3]{localsetapproximation} required existence of the  decomposition from \cite{badgerharmonicmeasure} as part of their hypotheses. \emph{By contrast, in this paper, we are able to establish the decomposition $A=A_1\cup\dots\cup A_d$ and obtain dimension estimates on the singular set $A\setminus A_1$ in Theorem \ref{t:main1} directly, without any reference to harmonic measure or dependence on \cite{badgerharmonicmeasure}.}
\end{remark}

Theorem \ref{t:main1} and \ref{t:main3} of the present paper yield several new pieces of information about the boundary of complimentary NTA domains with $\log f\in \VMO(d\omega^+)$, which we record in Theorem \ref{ourvmoresults}.

\begin{theorem}\label{ourvmoresults} Under the hypothesis of Theorem \ref{previousvmoresults}, the boundary $\partial\Omega=\Gamma_1\cup\dots\cup\Gamma_d$ satisfies the following additional properties.
\begin{enumerate}
\item For all $1\leq k\leq d$, the set $U_k:=\Gamma_1 \cup \ldots \cup \Gamma_k$ is relatively open in $\partial\Omega$ and $\Gamma_{k+1}\cup\dots\cup\Gamma_d$ is closed.
\item For all  $1\leq k\leq d$, $U_k$ is locally bilaterally well approximated by $\cH_{n,k}^{**}$.
\item For all $1\leq k\leq d$, $\partial\Omega$ is locally bilaterally well approximated along $\Gamma_k$ by $\cF_{n,k}^{**}$, i.e.~ $\limsup_{r\downarrow 0} \sup_{x\in K} \Theta_{\partial\Omega}^{\cF_{n,k}^{**}}(x,r)=0$ for every compact set $K\subseteq\Gamma_k$.
\item For all $1\leq l<k\leq d$, $U_l$ is relatively open in $U_k$ and $\Gamma_{l+1}\cup\dots\cup \Gamma_k$ is relatively closed in $U_k$.
\item $\partial\Omega\setminus\Gamma_1=\Gamma_2\cup\dots\cup\Gamma_d$ has upper Minkowski dimension at most $n-3$.
\item The ``even singular set" $\Gamma_2\cup\Gamma_4\cup\cdots$ has Hausdorff dimension at most $n-4$.
\item  When $n \geq 3$, the singular set $\partial\Omega\setminus\Gamma_1$ has Newtonian capacity zero.
\end{enumerate}
\end{theorem}

\begin{proof} Parts (i) and (iv) of the theorem are a direct consequence of Theorem \ref{t:main1}. Parts (ii) and (iii) follow from Theorem \ref{t:main1} in conjunction with Lemma \ref{puttingthestarin}, Theorem \ref{t:tangent-well}, and Remark \ref{closedclasses} (see the proof of Theorem \ref{t:main3}). Parts (v) and (vi) are a direct consequence of Theorem \ref{t:main3}. Newtonian capacity in $\RR^n$, $n\geq 3$, is precisely the Riesz $(n-2)$-capacity. Thus, part (vii) follows from (v) and the fact that sets of finite $s$-dimensional Hausdorff measure have Riesz $s$-capacity zero (see e.g.~ \cite[Chapter 4]{brownian} or \cite[Chapter 8]{Mattila95}). \end{proof}

\begin{remark} The dimension bounds (v) and (vi) in Theorem \ref{ourvmoresults} are sharp by example. See Remark \ref{r:sharp} and Remark \ref{closedclasses}. \end{remark}

\begin{remark} The fact that $\partial\Omega\setminus\Gamma_1$ has Newtonian capacity zero implies $\omega^\pm(\partial\Omega\setminus \Gamma_1)=0$; see \cite[Chapter 8]{brownian}.\end{remark}

\appendix

\section{Local set approximation}
\label{sect:2}

A general framework for describing bilateral and unilateral approximations of a set $A\subseteq\RR^n$ by a class $\cS$ of closed ``model" sets is developed in \cite{localsetapproximation}. In this appendix, we give a brief, self-contained abstract of the main definitions and theorems from this framework as used above, but refer the reader to \cite{localsetapproximation} for full details and further results. The principal results are two structure theorems for Reifenberg type sets; see Theorems \ref{t:open} and \ref{singapprox}.

\subsection{Distances between sets} If $A,B\subseteq\RR^n$ are nonempty sets, the \emph{excess of $A$ over $B$} is the asymmetric quantity defined by $\ex(A,B)=\sup_{a\in A}\inf_{b\in B}|a-b|\in[0,\infty]$. By convention, one also defines $\ex(\emptyset,B)=0$, but leaves $\ex(A,\emptyset)$ undefined. The excess is monotone, $$\ex(A,B)\leq \ex(A',B')\quad\text{whenever $A\subseteq A'$ and $B\supseteq B'$},$$ and satisfies the triangle inequality, $$\ex(A,C)\leq \ex(A,B)+\ex(B,C).$$ When $A=\{x\}$ for some $x\in\RR^n$, $\ex(\{x\},B)$ is usually called the \emph{distance of $x$ to $B$} and denoted by $\dist(x,B)$.

For all $x\in\RR^n$ and $r>0$, let $B(x,r)$ denote the open ball with center $x$ and radius $r$. (In \cite{localsetapproximation}, $B(x,r)$ denotes the closed ball, but see \cite[Remark 2.4]{localsetapproximation}.) For arbitrary sets $A,B\subseteq\RR^n$ with $B$ nonempty and for all $x\in\RR^n$ and $r>0$, define the \emph{relative excess of $A$ over $B$ in $B(x,r)$} by $$\mud{x}{r}(A,B)=r^{-1}\ex(A\cap B(x,r),B)\in [0,\infty).$$ Also, for all sets $A,B\subseteq\RR^n$ with $A$ and $B$ nonempty and for all $x\in\RR^n$ and $r>0$, define the \emph{relative Walkup-Wets distance between $A$ and $B$ in $B(x,r)$} by $$\mD{x}{r}[A,B]=\max\left\{\mud{x}{r}(A,B),\mud{x}{r}(B,A)\right\}\in[0,\infty).$$ Observe that $\mD{x}{r}[A,B]\leq 2$ if both $A\cap B(x,r)$ and $B\cap B(x,r)$ are nonempty; and $\mD{x}{r}[A,B]\leq 1$ if both $x\in A$ and $x\in B$.

\begin{lemma}[{\cite[Lemma 2.2, Remark 2.4]{localsetapproximation}}] \label{l:Dprop} Let $A,B,C\subseteq\RR^n$ be nonempty sets, let $x,y\in\RR^n$, and let $r,s>0$. \begin{itemize}
\item closure: $\mD{x}{r}[A,B] = \mD{x}{r}[A,\overline{B}] = \mD{x}{r}[\overline{A},\overline{B}] = \mD{x}{r}[\overline{A},B]$.
\item containment: $\mD{x}{r}[A,B]=0$ if and only if $\overline{A}\cap B(x,r)=\overline{B}\cap B(x,r)$.
\item quasimonotonicity: If $B(x,r)\subseteq B(y,s)$, then $\mD{x}{r}[A,B] \leq (s/r) \mD{y}{s}[A,B].$
\item strong quasitriangle inequality: If $\mud{x}{r}(A,B)\leq \varepsilon_1$ and $\mud{x}{r}(C,B)\leq \varepsilon_2$, then $$\mD{x}{r}[A,C] \leq (1+\varepsilon_2) \mD{x}{(1+\varepsilon_2)r}[A,B] + (1+\varepsilon_1) \mD{x}{(1+\varepsilon_1)r}[B,C].$$
\item weak quasitriangle inequalities: If $x\in B$, then $$\mD{x}{r}[A,C] \leq 2\mD{x}{2r}[A,B]+2\mD{x}{2r}[B,C].$$ If $B\cap B(x,r)\neq\emptyset$, then $$\mD{x}{r}[A,B]\leq 3\mD{x}{3r}[A,B]+3\mD{x}{3r}[B,C].$$
\item scale invariance: $\mD{x}{r}[A,B]=\mD{\lambda x}{\lambda r}[\lambda A,\lambda B]$ for all $\lambda>0$.
\item translation invariance: $\mD{x}{r}[A,B]=\mD{x+z}{r}[z+A,z+B]$ for all $z\in\RR^n$.
\end{itemize}
\end{lemma}

\begin{remark} The \emph{relative Hausdorff distance between $A$ and $B$ in $B(x,r)$},  defined by \begin{align*}
\D{x}{r}[A,B]=r^{-1}\max\big\{&\ex(A\cap B(x,r),B\cap B(x,r)),\\ &\qquad\qquad \ex(B\cap B(x,r),A\cap B(x,r))\big\}\end{align*} whenever $A\cap B(x,r)$ and $B\cap B(x,r)$ are both nonempty, is a common, better-known variant of the relative Walkup-Wets distance. We note that  $\mD{x}{r}[A,B]\leq \D{x}{r}[A,B]$ whenever both quantities are defined. Although the relative Hausdorff distance satisfies the triangle inequality rather than just the weak and strong quasitriangle inequalities enjoyed by the relative Walkup-Wets distance, the relative Hausdorff distance fails to be quasimonotone (see \cite[Remark 2.3]{localsetapproximation}).
This makes the relative Hausdorff distance unsuitable for use in the local set aproximation framework below. The use of the relative Walkup-Wets distance is deliberate and ensures that one can obtain structure theorems for Reifenberg type sets.\end{remark}

\subsection{Attouch-Wets topology, tangent sets, and pseudotangent sets}

Let $\CL{\RR^n}$ denote the collection of all nonempty closed sets in $\RR^n$. Let $\CL{0}$ denote the subcollection of all nonempty closed sets in $\RR^n$ containing the origin. We endow $\CL{\RR^n}$ and $\CL{0}$ with the \emph{Attouch-Wets topology} (see \cite[Chapter 3]{Beer} or \cite[Chapter 4]{RWbook}; i.e.~the topology described by the following theorem.

\begin{theorem}[{\cite[Theorem 2.5]{localsetapproximation}}] There exists a metrizable topology on $\CL{\RR^n}$ in which a sequence $(A_i)_{i=1}^\infty$ in $\CL{\RR^n}$ converges to a set $A\in\CL{\RR^n}$ if and only if $$ \lim_{i\rightarrow\infty} \ex(A_i\cap B(0,r),A)=0\quad\text{and}\quad \lim_{i\rightarrow\infty} \ex(A\cap B(0,r),A_i)=0\quad\text{for all }r>0.$$ Moreover, in this topology, $\CL{0}$ is sequentially compact; i.e.~ for any sequence $(A_i)_{i=1}^\infty$ in $\CL{0}$ there exists a subsequence $(A_{ij})_{j=1}^\infty$ and $A\in\CL{0}$ such that $(A_{ij})_{j=1}^\infty$ converges to $A$ in the sense above.
\end{theorem}

We write $A_i\rightarrow A$ or $A=\lim_{i\rightarrow\infty} A$ (in $\CL{\RR^n}$) to denote that a sequence of $(A_i)_{i=1}^\infty$ in $\CL{\RR^n}$ converges to a set $A\in\CL{\RR^n}$ in the Attouch-Wets topology. If each set $A_i\in\CL{0}$, then we may write $A_i\rightarrow A$ in $\CL{0}$ to emphasize that the limit $A\in\CL{0}$, as well.

\begin{lemma}[{\cite[Lemma 2.6]{localsetapproximation}}] Let $A,A_1,A_2,\dots\in\CL{\RR^n}$. The following statements are equivalent: \begin{enumerate}
\item $A_i\rightarrow A$ in $\CL{\RR^n}$;
\item $\lim_{i\rightarrow\infty} \mD{x}{r}[A_i,A]=0$ for all $x\in\RR^n$ and for all $r>0$;
\item $\lim_{i\rightarrow\infty} \mD{x_0}{r_j}[A_i,A]=0$ for some $x_0\in\RR^n$ and for some sequence $r_j\rightarrow\infty$.
\end{enumerate}
\end{lemma}

The notions of tangent sets and pseudotangent sets of a closed set in the following definition are modeled on notions of tangent measures (introduced by Preiss \cite{Preiss}) and pseudotangent measures (introduced by Kenig and Toro \cite{kenigtoroannals}) of a Radon measure.

\begin{definition}[{\cite[Definition 3.1]{localsetapproximation}}] \label{d:tangent} Let $T\in\CL{0}$, let $A\in\CL{\RR^n}$, and let $x\in A$. We say that $T$ is a \emph{pseudotangent set of $A$ at $x$} if there exist sequences $x_i\in A$ and $r_i>0$ such that $x_i\rightarrow x$, $r_i\rightarrow 0$, and $$\frac{A-x_i}{r_i}\rightarrow T\quad\text{in }\CL{0}.$$ If $x_i=x$ for all $i$, then we call $T$ a \emph{tangent set of $A$ at $x$}. Let $\PsTan(A,x)$ and $\Tan(A,x)$ denote the collections of all pseudotangent sets of $A$ at $x$ and all tangent sets of $A$ at $x$, respectively.
\end{definition}

\begin{lemma}[{\cite[Remark 3.3, Lemma 3.4, Lemma 3.5]{localsetapproximation}}] $\Tan(A,x)$ \label{l:tan-properties} and $\PsTan(A,x)$ are closed in $\CL{0}$ and are nonempty for all $A\in\CL{\RR^n}$ and $x\in A$. Moreover, \begin{itemize}
\item If $T\in\Tan(A,x)$ and $\lambda>0$, then $\lambda A\in\Tan(A,x)$.
\item If $T\in\PsTan(A,x)$ and $\lambda>0$, then $\lambda T\in\PsTan(A,x)$.
\item If $T\in\PsTan(A,x)$ and $y\in T$, then $T-y\in\PsTan(A,x)$.\end{itemize}\end{lemma}

\subsection{Reifenberg type sets and Mattila-Vuorinen type sets}

\begin{definition}[{\cite[Definition 4.1 and Definition 4.7]{localsetapproximation}}] \label{d:big} Let $A\subseteq\RR^n$ be nonempty. \begin{enumerate}
\item A \emph{local approximation class} $\cS$ is a nonempty collection of closed sets in $\CL{0}$ such that $\cS$ is a cone; that is, for all $S\in\cS$ and $\lambda>0$, $\lambda S\in\cS$.
\item For every $x\in\RR^n$ and $r>0$, define the \emph{bilateral approximability} $\Theta_A^{\cS}(x,r)$ of $A$ by $\cS$ at location $x$ and scale $r$ by $$\Theta_A^{\cS}=\inf_{S\in\cS} \mD{x}{r}[A,x+S]\in[0,\infty).$$
\item We say that $x\in A$ is an \emph{$\cS$ point} of $A$ if $\lim_{r\downarrow 0} \Theta_A^{\cS}(x,r)=0$.
\item We say that $A$ is \emph{locally bilaterally $\varepsilon$-approximable by $\cS$} if for every compact set $K\subseteq A$ there exists $r_K$ such that $\Theta_A^{\cS}(x,r)\leq \varepsilon$ for all $x\in K$ and $0<r\leq r_K$.
\item We say that $A$ is \emph{locally bilaterally well approximated by $\cS$} if $A$ is locally bilaterally $\varepsilon$-approximable by $\cS$ for all $\varepsilon>0$.
\item For every $x\in\RR^n$ and $r>0$, define the \emph{unilateral approximability} $\beta_A^{\cS}(x,r)$ of $A$ by $\cS$ at location $x$ and scale $r$ by $$\beta_A^{\cS}(x,r)=\inf_{S\in\cS} \mud{x}{r}(A,x+S)\in[0,1].$$
\item We say that $A$ is \emph{locally unilaterally  $\varepsilon$-approximable by $\cS$} if for every compact set $K\subseteq A$ there exists $r_K$ such that $\beta_A^{\cS}(x,r)\leq \varepsilon$ for all $x\in K$ and $0<r\leq r_K$.
\item We say that $A$ is \emph{locally unilaterally well approximated by $\cS$} if $A$ is locally unilaterally $\varepsilon$-approximable by $\cS$ for all $\varepsilon>0$.
\end{enumerate}
\end{definition}

\begin{remark}Sets that are bilaterally approximated by $\cS$ are called \emph{Reifenberg type sets} and sets that are unilaterally approximated by $\cS$ are called \emph{Mattila-Vuorinen type sets} with deference to pioneering work of Reifenberg \cite{Reifenberg} and Mattila and Vuorinen \cite{MV}, which investigated, respectively,  regularity of sets that admit locally uniform \emph{bilateral} and \emph{unilateral} approximations by $\cS=G(n,m)$, the Grassmannian of $m$-dimensional subspaces of  $\RR^n$. The concept of (unilateral) approximation numbers first appeared in the work of Jones \cite{JonesTSP} in connection with the Analyst's traveling salesman theorem. For additional background, including examples of Reifenberg type sets that have appeared in the literature, see the introduction of \cite{localsetapproximation}. \end{remark}

\begin{remark} For any nonempty closed set $A\subseteq\RR^n$ and point $x\in A$, the set $\Tan(A,x)$ of tangent sets of $A$ at $x$ and the set $\PsTan(A,x)$ of pseudotangent sets of $A$ at $x$ are local approximation classes by Lemma \ref{l:tan-properties}. We also note that from the definitions, it is immediate that any set $A\subseteq\RR^n$ which is locally bilaterally well approximated by some local approximation class $\cS$ is also locally unilaterally well approximated by $\cS$.  \end{remark}

The following essential properties of bilateral approximation numbers appear across a number of lemmas in \cite[\S4]{localsetapproximation}, which we consolidate into a single theorem statement. See \cite[Lemma 7.2]{localsetapproximation} for the analogous properties of unilateral approximation numbers.

\begin{lemma}[{\cite[\S4, Remark 2.4]{localsetapproximation}}] \label{l:tprop} Let $\cS$ be a local approximation class, let $A\subseteq\RR^n$ be nonempty, let $x,y\in\RR^n$, and let $r,s>0$. \begin{itemize}
\item size: $0 \leq \Theta_A^{\cS}(x,r) - \dist(x,A)/r \leq 1$; thus, $0\leq \Theta^{\cS}_A(x,r)\leq 1$ for all $x\in A$.
\item scale invariance: $\Theta_A^{\cS}(x,r)=\Theta_{\lambda A}^{\cS}(\lambda x,\lambda r)$ for all $\lambda>0$.
\item translation invariance: $\Theta_A^{\cS}(x,r)=\Theta_{A+z}^{\cS}(x+z,r)$ for all $z\in\RR^n$.
\item closure: $\Theta_A^{\cS}(x,r)=\Theta_{\overline{A}}^{\cS}(x,r)$
\item quasimonotonicity: If $B(x,r)\subseteq B(y,s)$ and $|x-y|\leq ts$, then $$\Theta_A^{\cS}(x,r) \leq \frac{s}{r}\left[t+(1+t)\Theta_A^{\cS}(y,(1+t)s)\right].$$ In particular, if $r<s$, then $\Theta_A^{\cS}(x,r) \leq (s/r) \Theta_A^{\cS}(x,s)$.
\item limits: If $A,A_1,A_2,\dots\in \CL{\RR^n}$ and $A_i\rightarrow A$ in $\CL{\RR^n}$, then \begin{equation*}\begin{split}\label{betaLimits.1}\frac{1}{1+\varepsilon} \limsup_{i\to\infty}\,&\Theta^\cS_{A_i}\left(x,\frac{r}{1+\varepsilon}\right)\\ \leq &\,\Theta^\cS_{A}(x,r) \leq
(1+\varepsilon) \liminf_{i\to\infty}\Theta^\cS_{A_i}(x,r(1+\varepsilon))\quad\text{for all }\varepsilon>0.\end{split}.\end{equation*}
\end{itemize}
 \end{lemma}

The notions of $\cS$ points and locally bilaterally and unilaterally well approximated sets admit the following characterizations in terms of tangent sets and pseudotangent sets. Here $\ccS$ denotes the closure of $\cS$ in $\CL{0}$ with respect to the Attouch-Wets topology.

\begin{theorem}[{\cite[Corollary 4.12, Corollary 4.15, Lemma 7.7, Theorem 7.10]{localsetapproximation}}] \label{t:tangent-well} Let $\cS$ be a local approximation class and let $A\subseteq\RR^n$ be a nonempty set and let $x_0\in A$. Then \begin{enumerate}
\item $x_0$ is an $\cS$ point of $A$ if and only if $\Tan(\overline{A},x_0)\subseteq \ccS $;
\item $A$ is locally bilaterally well approximated by $\cS$ if and only if $$\PsTan(\overline{A},x)\subseteq \ccS\quad\text{for all }x\in A;$$
\item $A$ is locally unilaterally well approximated by $\cS$ if and only if $$\PsTan(\overline{A},x)\subseteq\{T\in\CL{0}:T\subseteq S\text{ for some } S\in\ccS\}\quad\text{for all }x\in A.$$
\end{enumerate}
\end{theorem}

\subsection{Detectability and structure theorems for Reifenberg type sets}

\begin{definition}[{\cite[Definition 5.8]{localsetapproximation}}] \label{d:Tpdp} Let $\cT$ and $\cS$ be local approximation classes. We say that \emph{$\cT$ points are detectable in $\cS$} if there exist a constant $\phi>0$ and a function $\Phi:(0,1)\rightarrow(0,\infty)$ with $\liminf_{s\rightarrow 0+}\Phi(s)=0$ such that if $S\in\cS$ and $\Theta^\cT_S(0,r)<\phi$, then $\Theta^\cT_S(0,sr)<\Phi(s)$ for all $s\in(0,1)$. To emphasize a choice of $\phi$ and $\Phi$, we may say that \emph{$\cT$ points are $(\phi,\Phi)$ detectable in $\cS$}. \end{definition}

\begin{definition}[{\cite[Definition 5.1]{localsetapproximation}}] \label{def-perp}Let $\cT$ be a local approximation class. The \emph{bilateral singular class of $\cT$} is the local approximation class $\cT^\perp$ given by $$\cT^\perp=\{Z\in\CL{0}:\liminf_{r\downarrow 0}\Theta_Z^{\cT}(0,r)>0\}=\{Z\in\CL{0}:\Tan(Z,0)\cap \overline{\cT}=\emptyset\}.$$
\end{definition}

The following structure theorem decomposes a set $A\subseteq\RR^n$ that is locally bilaterally well approximated by $\cS$ into an open ``regular part" $A_\ccT$ and closed ``singular part" $A_\cTp$, on the condition that ``regular" $\cT$ points are detectable in $\cS$.

\begin{theorem}[{\cite[Theorem 6.2, Corollary 6.6, Corollary 5.12]{localsetapproximation}}]\label{t:open} Let $\cT$ and $\cS$ be local approximation classes. Suppose $\cT$ points are $(\phi,\Phi)$ detectable in $\cS$.
If $A\subseteq\RR^n$ is locally bilaterally well approximated by $\cS$, then $A$ can be written as a disjoint union \begin{equation*} A=A_\ccT\cup A_\cTp \quad (A_\ccT\cap A_\cTp=\emptyset),\end{equation*} where \begin{enumerate}
\item $\PsTan(\cl{A},x)\subseteq \ccS\cap\ccT$ for all $x\in A_\ccT$, and
\item $\Tan(\cl{A},x)\subseteq\ccS\cap \cTp=\{S\in\ccS:\Theta_S^{\cT}(0,r)\geq \phi\text{ for all }r>0\}$ for all $x\in A_\cTp$.\end{enumerate}
Moreover: \begin{enumerate}
\item[(iii)] $A_\ccT$ is relatively open in $A$ and $A_\ccT$ is locally bilaterally well approximated by $\cT$.
\item[(iv)] $A$ is locally bilaterally well approximated \textbf{along} $A_{\cTp}$ by $\ccS\cap \cTp$ in the sense that $\limsup_{r\downarrow 0}\sup_{x\in K}\Theta_A^{\ccS\cap \cTp}(x,r)=0$ for all compact sets $K\subseteq A_\cTp$.
\end{enumerate}
\end{theorem}

\begin{remark} \label{r:detectx} Suppose $\cT$ points are $(\phi,\Phi)$ detectable in $\cS$ and $A$ is locally bilaterally well approximated by $\cS$. From the proof that $A_\ccT$ is open in the proof of \cite[Theorem 6.2]{localsetapproximation}, there exist constants $\alpha,\beta>0$ depending only on $\phi$ and $\Phi$ such that if $\Theta_A^{\cS}(x,r')<\alpha$ for all $0<r'\leq r$ and $\Theta_A^{\cT}(x,r)<\beta$ for some $x\in A$ and $r>0$, then $x\in A_{\ccT}$.\end{remark}

A local approximation class $\cS$ is called \emph{translation invariant} if for all $S\in\cS$ and $x\in S$, $S-x\in \cS$. It is an exercise to show that if $\cS$ is translation invariant, then its closure $\ccS$ is translation invariant, as well. If $\cT$ and $\cS$ are local approximation classes such that \begin{equation} \label{e:rs} \hbox{$\ccS$ is translation invariant, and $\cT$ points are $(\phi,\Phi)$ detectable in $\cS$,}\end{equation} then every set $X\in\ccS$ is locally (in fact, globally) bilaterally well approximated by $\cS$, whence $X=X_{\ccT}\cup X_\cTp$ and $X_\cTp$ is closed (since $X$ is closed) by Theorem \ref{t:open}.

\begin{definition}[{\cite[Definition 7.12]{localsetapproximation}}] \label{sing} Let $\cT$ and $\cS$ be local approximation classes. Assume (\ref{e:rs}). We define the local approximation class of \emph{$\cT$ singular parts of sets in $\ccS$} by
$\sing\cT\ccS = \{X_\cTp: X\in \ccS \text{ and } 0\in X_\cTp\}.$
\end{definition}

\begin{theorem}[{\cite[Theorem 7.14]{localsetapproximation}}]\label{singapprox} Let $\cT$ and $\cS$ be local approximation classes. Assume (\ref{e:rs}). If $A\subseteq\RR^n$ is locally bilaterally well approximated by $\cS$, then $A_\cTp$ is locally unilaterally well approximated by $\sing\cT\ccS$.
\end{theorem}

\subsection{Covering profiles and dimension bounds for Mattila-Vuorinen type sets}

Finally, we record two upper bounds on the dimension of sets that are locally unilaterally well approximated by a local approximation class $\cS$ with a uniform covering profile. Additional quantitative bounds for locally unilaterally $\varepsilon$-approximable sets may be found in \cite[\S8]{localsetapproximation}.

For reference, let us  recall a definition of Minkowski dimension; e.g., see \cite{Mattila95}.

\begin{definition}Let $A\subseteq\RR^n$, let $x\in\RR^n$, and let $r,s>0$. The \emph{(intrinsic) $s$-covering number of $A$} is defined by \begin{equation*}\covplain(A,s):=\min\left\{k\geq 0:A\subseteq \bigcup_{i=1}^k \ball(a_i,s)\text{ for some }a_i\in A\right\}.\end{equation*}
For bounded sets $A\subseteq\RR^n$, the \emph{upper Minkowski dimension of $A$} is given by \begin{equation*}
 \udim_M(A)=\limsup_{s\downarrow 0} \frac{\log\left(\covplain(A,s)\right)}{\log(1/s)}.
\end{equation*} For unbounded sets $A\subseteq\RR^n$, the \emph{upper Minkowski dimension of $A$} is given  by $$\udim_M(A)=\lim_{t\uparrow\infty} \left(\udim_M A\cap \ball(0,t)\right).$$\end{definition}

Letting $\dim_H(A)$ denote the usual \emph{Hausdorff dimension} of a set $A\subseteq\RR^n$, $$0\leq \dim_H(A)\leq \udim_M(A)\leq n\quad\text{for all }A\subseteq\RR^n,$$ with $\dim_H(A)<\udim_M(A)$ for certain sets. For the definition of Hausdorff dimension, several equivalent definitions of Minkowski dimension, and related results, we refer the reader to Mattila \cite{Mattila95}.

\begin{definition}[{\cite[Definition 8.2 and 8.4]{localsetapproximation}}]\label{coveringprofiles} Let $\cS$ be a local approximation class. We say that \emph{$\cS$ has an $(\alpha, C, s_0)$ covering profile} for some $\alpha>0$, $C>0$, and $s_0\in(0,1]$ provided $\covplain(S\cap B(0,r),sr) \leq Cs^{-\alpha}$ for all $S\in\cS$, $r>0$, and $s\in (0,s_0]$.
\end{definition}

\begin{theorem}[{\cite[Corollary 8.9]{localsetapproximation}}]\label{c:dim3} Let $\cS$ be a local approximation class such that $\cS$ has an $(\alpha, C, s_0)$ covering profile. If  $A\subseteq\RR^n$ is closed and $A$ is locally unilaterally well approximated by $\cS$, then $\udim_M(A)\leq \alpha$.
\end{theorem}

\begin{theorem}[{\cite[Corollary 8.12]{localsetapproximation}}] \label{c:dim5} Let $\cS$ be a local approximation class such that $\cS$ has an $(\alpha,C,s_0)$ covering profile. If the subspace topology on $A\subseteq\RR^n$ is $\sigma$-compact and $A$ is locally unilaterally well approximated by $\cS$, then $\dim_H (A)\leq \alpha$.
\end{theorem}

\section{Limits of complimentary NTA domains}

For reference, let us recall that a connected open set $\Omega \subset \mathbb R^n$ is called an \emph{NTA domain} (see Definition \ref{ntadomains} and Remark \ref{r:ntadomains}) if there exist constants $M>1$ and $R > 0$
 for which the following are true:
\begin{enumerate}
\item $\Omega$ satisfies the \emph{corkscrew condition}: for all $Q \in \partial \Omega$ and $0 < r < R$, there exists $x\in \Omega\cap B(Q,r)$ such that $\dist(x,\partial\Omega)>M^{-1}r$.
\item $\RR^n\setminus\Omega$ satisfies the corkscrew condition.
\item $\Omega$ satisfies the \emph{Harnack chain condition}: If $x_1, x_2\in \Omega\cap B(Q,r/4)$ for some $Q\in \partial \Omega$ and $0<r<R$, and $\dist(x_1,\partial\Omega)>\delta$, $\dist(x_2,\partial\Omega)>\delta$, and $|x_1-x_2|<2^l\delta$ for some $\delta>0$ and $l\geq 1$, then there exists a chain of no more than $Ml$ overlapping balls  connecting $x_1$ to $x_2$ in $\Omega$ such that for each ball $B=B(x,s)$ in the chain: \begin{align*}&M^{-1}s<\gap(B,\partial\Omega)<Ms, &&\gap(B,\partial\Omega)=\inf_{x\in B}\inf_{y\in\partial\Omega}|x-y|, \\ \diam B> &M^{-1}\min\{\dist(x_1,\partial\Omega),\dist(x_2,\partial\Omega)\}, &&\diam B=\sup_{x,y\in B}|x-y|.\end{align*}
\end{enumerate} The constants $M$ and $R$ are called \emph{NTA constants} of $\Omega$, and the value $R=\infty$ is allowed when $\partial\Omega$ is unbounded. Lemma \ref{closedunderblowup} asserts that if $\RR^n\setminus \Gamma_i=\Omega^+_i\cup\Omega^-_i$, where $\Omega^+_i$ and $\Omega^-_i$ are complimentary NTA domains with NTA constants $M$ and $R$ independent of $i$, and $\Gamma_i\rightarrow\Gamma\neq\emptyset$ in the Attouch-Wets topology, then $\RR^n\setminus\Gamma=\Omega^+\cup\Omega^-$, where $\Omega^+$ and $\Omega^-$ are complimentary NTA domains with constants $2M$ and $R$.

\begin{proof}[Proof of Lemma \ref{closedunderblowup}]
Assume that we are given a sequence $(\Gamma_i,\Omega_i^+,\Omega_i^-)$, constants $M$ and $R$, and a set $\Gamma$ satisfying the hypothesis of the lemma. We note and will frequently use below that $\RR^n\setminus\overline{\Omega_i^\pm}=\Omega_i^\mp$, $\Gamma_i=\partial\Omega^\pm_i$, and $\RR^n=\Omega^+_i\cup \Gamma_i\cup\Omega^-_i$ by the separation condition on $\Gamma_i$ and the corkscrew conditions for $\Omega_i^\pm$.

\emph{Step 0 (Definition of $\Omega^+$ and $\Omega^-$).} Because the sequence $(\Gamma_i)_{i=1}^\infty$ does not escape to infinity (as $\Gamma_i\rightarrow\Gamma$), neither do $(\overline{\Omega^\pm_i})_{i=1}^\infty$. Thus, there is a subsequence of $(\Gamma_i,\Omega_i^+,\Omega_i^-)$ (which we relabel) and nonempty closed sets $F^+,F^-\subseteq\RR^n$ such that $\overline{\Omega_i^\pm}\rightarrow F^\pm$. Here and below, convergence of a sequence of nonempty closed sets in $\RR^n$ is always taken with respect to the Attouch-Wets topology; we refer the reader to  \S\S A.1 and A.2 above for a brief introduction to this topology and to \cite[Chapter 4]{RWbook} or \cite[Chapter 3]{Beer} for the rest of the  story. Consider the open sets $\Omega^+$ and $\Omega^-$ defined by  $$\Omega^+ = \RR^n\setminus F^-\quad\text{and}\quad \Omega^- = \RR^n\setminus F^+.$$ We will show that $\RR^n\setminus \Gamma=\Omega^+\cup\Omega^-$ and $\Omega^+$ and $\Omega^-$ are complementary NTA domains with NTA constants $2M$ and $R$.

\emph{Step $\frac12$ ($\Omega^+$, $\Gamma$, and $\Omega^-$ are disjoint).} First, because  $\Gamma_i\subseteq\overline{\Omega^\pm_i}$ for all $i\geq 1$, $\Gamma_i\rightarrow\Gamma$, and $\overline{\Omega^\pm_i}\rightarrow F^\pm$, we have $\Gamma\subseteq F^\pm$, as well. Hence, by definition of $\Omega^\pm$, $$\Gamma\cap \Omega^\pm\subseteq F^\mp\cap \Omega^\pm=F^\mp\setminus F^\mp = \emptyset.$$
Next, suppose that $x\in \Omega^\pm$. Then $x\not\in F^\mp$, whence $\dist(x,F^\mp)=\delta$ for some $\delta>0$. Since $\overline{\Omega^\mp_i}\rightarrow F^\mp$, it follows that $\dist(x,\overline{\Omega^\mp_i})\geq \delta/2$ for all $i\gg 1$. In particular, $x\in \Omega^\pm_i \subseteq\overline{\Omega^\pm_i}$ for all $i\gg 1$, because $\RR^n\setminus\overline{\Omega_i^\mp}=\Omega_i^\pm$. Since $\overline{\Omega^\pm_i}\rightarrow F^\pm$, we obtain $x\in F^\pm$.  Thus, $x\not\in \Omega^\mp$ whenever $x\in\Omega^\pm$. We conclude that $\Omega^+\cap\Omega^-=\emptyset$.

\emph{Step 1 ($\RR^n=\Omega^+\cup\Gamma\cup\Omega^-$).} Let $x\in\RR^n$. Because $\RR^n=\overline{\Omega_i^+}\cup\overline{\Omega_i^-}$, at least one of the following alternatives occur: $x\in \overline{\Omega^+_i}$ for infinitely many $i$ or $x\in \overline{\Omega^-_i}$ for infinitely many $i$. Hence $x\in F^+$ or $x\in F^-$, since $\overline{\Omega_i^+}\rightarrow F^+$ and $\overline{\Omega_i^-}\rightarrow F^-$. As $x$ was arbitrary, we have $$\RR^n=F^+\cup F^-=(F^+\setminus F^-)\cup (F^+\cap F^-)\cup (F^-\setminus F^+)=\Omega^+\cup(F^+\cap F^-)\cap \Omega^-.$$ Therefore, as soon as we show that $F^+\cap F^-=\Gamma$, we will have  $\RR^n=\Omega^+\cup\Gamma\cup\Omega^+$.

To prove that $F^+\cap F^-\subseteq\Gamma$, suppose that $y\in F^+\cap F^-$. Since $\overline{\Omega^\pm_i}\rightarrow F^\pm$, we can locate points $y^\pm_i\in\overline{\Omega^\pm_i}$ such that $y^\pm_i\rightarrow y$. The line segment between $y^+$ and $y^-$ must intersect $\Gamma_i=\overline{\Omega^+_i}\cap \overline{\Omega^-_i}$, say $Q_i\in [y_i^+,y_i^-]\cap \Gamma_i$. Then $Q_i\rightarrow y$, and because $\Gamma_i\rightarrow \Gamma$, we obtain $y\in \Gamma$. Thus, $F^+\cap F^-\subseteq\Gamma$.

To prove that $\Gamma\subseteq F^+\cap F^-$, suppose that $z\in \Gamma$. Since $\Gamma_i\rightarrow\Gamma$, there exists $z_i\in\Gamma_i$ such that $z_i\rightarrow \Gamma$. Because $\Gamma_i=\partial\Omega^+=\partial\Omega^-$, we can locate points $z_i^\pm\in \Omega^\pm_i \cap B(z_i,1/i)$. Then $z_i^\pm\rightarrow z$, and because $\overline{\Omega^\pm_i}\rightarrow F^\pm$, we obtain $z\in F^+\cap F^-$. Thus, $\Gamma\subseteq F^+\cap F^-$.

\emph{Step $\frac32$ ($\partial\Omega^\pm\subseteq \Gamma$).} Since $\Omega^+$ and $\Omega^-$ are open and disjoint (Steps 0 and $\frac12$), $\Omega^\pm$ coincides with the interior of $\Omega^\pm$ and $\Omega^\mp$ is contained in the exterior of $\Omega\pm$. Therefore, the boundary of $\Omega^\pm$ must be contained in $\RR^n\setminus(\Omega^\pm\cup\Omega^\mp)=\Gamma$ (Step 1).

\emph{Step 2 (Corkscrew condition for $\Omega^\pm$).} Suppose that $Q\in\partial\Omega^\pm$ and $0<r<R$. By Step $\frac32$, $Q\in\Gamma$. Since $\Gamma_i\rightarrow\Gamma$, there exists $Q_i\in\Gamma_i=\partial\Omega^\pm_i$ such that $Q_i\rightarrow Q$. By the corkscrew condition for $\Omega_i^\pm$, there exists a point $y^\pm_i\in\Omega^\pm_i\cap B(Q_i,\tfrac34 r)$ such that $$\dist(y^\pm_i,\overline{\Omega^\mp_i})=\dist(y^\pm_i,\partial\Omega^\pm_i)> \tfrac34r/M.$$ Assume $i\geq 1$ is sufficiently large such that $$y^\pm_i\in B(Q_i,\tfrac34r)\subset B(Q,\tfrac45r)\quad\text{and}\quad\dist(y_i^\pm, F^\mp)\leq |y_i^\pm-Q|<\tfrac45r.$$ Then $\dist(y_i^\pm, F^\mp)=\dist(y_i^\pm,F^\mp\cap B(Q,\tfrac45r))$. Hence, by the triangle inequality for excess, \begin{equation*}\begin{split}\dist(y_i^\pm,\overline{\Omega_i^\mp}) &\leq \dist(y_i^\pm, F^\mp\cap B(Q,\tfrac45 r)) + \ex(F^\mp\cap B(Q,\tfrac45r),\overline{\Omega^\mp_i})
\\ &=\dist(y_i^\pm, F^\mp) + \ex(F^\mp\cap B(Q,\tfrac45r),\overline{\Omega^\mp_i}).\end{split}\end{equation*} The last term vanishes as $i\rightarrow\infty$, since $\overline{\Omega^\mp_i}\rightarrow F^\mp$ in the Attouch-Wets topology. Thus, \begin{equation}\label{e:cks1}\dist(y_i^\pm,F^\mp)\geq \dist(y_i^\pm,\overline{\Omega^\mp_i})-\ex(F^\mp\cap B(Q,\tfrac45r),\overline{\Omega^\mp_i})> \tfrac23r/M\quad\text{for all }i\gg 1.\end{equation} By compactness, we can choose subsequences $(y^\pm_{ij})_{j=1}^\infty$ of $(y^\pm_i)_{i=1}^\infty$ such that $y^\pm_{ij}\rightarrow y^\pm$ for some $y^\pm\in \overline{B(Q,\tfrac45r)}\subset B(Q,r)$. By \eqref{e:cks1}, it follows that $\dist(y^\pm, F^\mp)\geq \frac23 r/M>\frac12 r/M.$ Thus, $y^\pm \in \Omega^\pm\cap B(Q,r)$ and $$\dist(y^\pm, \partial\Omega^\pm)=\dist(y^\pm,F^\mp)>\tfrac{1}{2}r/M.$$ Therefore, $\Omega^\pm$ satisfies the corkscrew condition with constants $2M$ and $R$. We note that by an obvious modification of the argument, one can show that $\Omega^\pm$ satisfies the corkscrew condition with constants $M'$ and $R$ for all $M'>M$.

\emph{Step $\frac52$ ($\partial\Omega^\pm=\Gamma$).} By Step $\frac32$, $\partial\Omega^\pm\subseteq\Gamma$. To see that $\Gamma\subseteq\partial\Omega^\pm$, suppose that $Q\in\Gamma$. By the proof of Step 2, the ball $B(Q,r)$ contains points in $\Omega^\pm$ for all $0<r<R$. Because $\Omega^\mp$ is disjoint from $\Omega^\pm$, it follows that $Q\in\partial\Omega^\pm$. We conclude that $\partial\Omega^\pm=\Gamma$.

\emph{Step 3 (Harnack chain condition for $\Omega^\pm$).} Assume that $x_1,x_2\in\Omega^\pm\cap B(Q,r/4)$ for some $Q\in\Gamma=\partial\Omega^\pm$ and $0<r<R$. Furthermore, assume that  $\delta_1:=\dist(x_1,\partial\Omega)>\delta$, $\delta_2:=\dist(x_2,\partial\Omega)>\delta$, and $|x_1-x_2|<2^l\delta$ for some $\delta>0$ and $l\geq 1$. We must show that $x_1$ can be connected to $x_2$ in $\Omega^\pm$ by a ``short" chain of balls in $\Omega^\pm$ remaining ``far away" from the boundary $\partial\Omega^\pm$, or equivalently, remaining ``far away" from $F^\mp$. Since $\Gamma_i\rightarrow\Gamma$, there exists $Q_i\in\Omega^\pm_i$ such that $Q_i\rightarrow Q$. Because $\overline{\Omega^\mp_i}\rightarrow F^\mp$ in the Attouch-Wets topology, for all $i\geq 1$ sufficiently large, $r(1+|Q-Q_i|)<R$, $x_1, x_2\in \Omega_i^\pm\cap B(Q_i,r(1+|Q-Q_i|)/4)$, and \begin{equation*}\begin{split} \dist(x_1,\partial\Omega_i^\pm)&>\delta_1/2>\delta/2,\quad
\dist(x_2,\partial\Omega_i^\pm) > \delta_2/2>\delta/2,\\ &\text{and}\quad |x_1-x_2| < 2^l\delta=2^{l+1}\delta/2.\end{split}\end{equation*} (The details are similar to those written in the proof of the corkscrew condition in Step 2.) By the Harnack chain condition for $\Omega^\pm_i$, we can find a chain of no more than $M(l+1)\leq 2Ml$ balls connecting $x_1$ to $x_2$ in $\Omega^\pm_i$ such that for each ball $B=B(x,s)$ in the chain, $$M^{-1}s<\gap(B,\partial\Omega^\pm_i)<Ms$$ and $$\diam B> M^{-1}\min\{\dist(x_1,\partial\Omega^\pm_i),\dist(x_2,\partial\Omega^\pm_i)\}.$$ Since $\overline{\Omega_i^\mp}\rightarrow F^\mp$ in the Attouch-Wets topology, it follows that for all sufficiently large $i$, $$(2M)^{-1}s<\gap(B,\partial\Omega^\pm)<2Ms$$ and $$ \diam B> (2M)^{-1}\min\{\dist(x_1,\partial\Omega^\pm),\dist(x_2,\partial\Omega^\pm)\}.$$ (Again, the details are similar to those in Step 2.) By the gap condition, we also know  each ball in the chain belongs to $\Omega^\pm$. Therefore, $\Omega^\pm$ satisfies the Harnack chain condition with constants $2M$ and $R$. We remark that given the discrete nature of the constant in the Harnack chain condition (counting balls), we cannot expect to be able to replace $2M$ by $\lambda M$ for arbitrary $\lambda>1$.

\emph{Step 4 ($\Omega^+$ and $\Omega^-$ are connected).} It is well known that every NTA domain is a uniform domain with constants that depend only on the interior corkscrew condition and Harnack chain condition; e.g., see \cite[Theorem 2.15]{chord-arc-5}. Explicitly, this means that for every $M>1$ and $R>0$, there exists $C>1$ and $c\in(0,1)$ such that for every NTA domain $\Omega\subseteq\RR^n$ with NTA constants $M$ and $R$, and for every $x_0,x_1\in\Omega$, there exists a continuous path $\gamma:[0,1]\rightarrow\Omega$ such that $\gamma(0)=x_0$, $\gamma(1)=x_1$, $\mathrm{length}(\gamma)\leq C|x_0-x_1|$, and $\dist(\gamma(t),\partial\Omega)\geq c\min\{\dist(x_0,\partial\Omega),\dist(x_1,\partial\Omega)\}$ for all $t\in[0,1]$.

Let $x_0$ and $x_1$ be arbitrary distinct points in $\Omega^\pm$, and set $$\delta=\min\{\dist(x_0,\partial\Omega^\pm),\dist(x_1,\partial\Omega^\pm)\}
=\min\{\dist(x_0,F^\mp),\dist(x_1,F^\mp)\}.$$ Assign $B=B(x_0,3C|x_0-x_1|+3\delta)$, where $C$ is the constant from the previous paragraph. Note that $B$ contains $x_0$, $x_1$, and every path passing through $x_0$ of length no greater than $C|x_0-x_1|$, and the closest point in $F^\mp$ for each item listed above, with room to spare. Since $\overline{\Omega^\mp_i}\rightarrow F^\mp$ in the Attouch-Wets topology,  \begin{equation}\label{e:Fex} \ex(\overline{\Omega^\mp_i}\cap B,F^\mp)< \tfrac13c\delta\quad\text{and}\quad
 \ex(F^\mp\cap B,\overline{\Omega^\mp_i})<\tfrac13c\delta\quad\text{for all }i\gg 1,\end{equation} where $c$ is the constant from the previous paragraph. Pick any $i$ such that \eqref{e:Fex} holds. Then $\dist(x_0,\overline{\Omega^\mp_i})\geq (1-c/3)\delta> \frac23\delta$ and $\dist(x_1,\overline{\Omega^\mp_i})\geq (1-c/3)\delta>\frac23\delta$. In particular, $x_0,x_1\in\Omega^\pm_i$ and $\min\{\dist(x_0,\partial\Omega^\pm_i),\dist(x_1,\partial\Omega^\pm_i)\}>\frac23\delta$. Since $\Omega^\pm_i$ is an NTA domain with NTA constants $M$ and $R$, by the previous paragraph we can find a continuous path $\gamma:[0,1]\rightarrow\Omega^\pm_i$ such that $\gamma(0)=x_0$, $\gamma(1)=x_1$, $\mathrm{length}(\gamma)\leq C|x_0-x_1|$, and $\dist(\gamma(t),\overline{\Omega^\mp_i})=\dist(\gamma(t),\partial\Omega^\pm_i)>\frac23 c\delta$ for all $t\in[0,1]$. Using \eqref{e:Fex} once again, we obtain $\dist(\gamma(t),F^\mp)>\frac13c\delta$ for all $t\in[0,1]$. In particular, $\gamma(t)\in\Omega^\pm$ for all $t\in[0,1]$. Thus, $\gamma$ is a continuous path joining $x_0$ to $x_1$ inside the set $\Omega^\pm$. Since $x_0$ and $x_1$ were fixed arbitrarily, we conclude that $\Omega^\pm$ is connected.

\emph{Conclusion.} We have shown that $\RR^n\setminus\Gamma=\Omega^+\cup\Omega^-$ (Step 1), where $\Omega^+$ and $\Omega^-$ are open (Step 0), connected (Step 4), and satisfy  corkscrew (Step 2) and Harnack chain conditions (Step 3) with constants $2M$ and $R$. Therefore, $\RR^n\setminus \Gamma=\Omega^+\cup\Omega^-$ is the union of complimentary NTA domains $\Omega^+$ and $\Omega^-$ with NTA constants $2M$ and $R$, as desired.
\end{proof}

\bibliography{bet-hsing}{}
\bibliographystyle{amsalpha}

\end{document}